\AtBeginDocument{%
  \paperwidth=\dimexpr
    1in + \oddsidemargin
    + \textwidth
    + 1in + \oddsidemargin
  \relax
  \paperheight=\dimexpr
    1in + \topmargin
    + \headheight + \headsep
    + \textheight
    + 1in + \topmargin
  \relax
  \usepackage[pass]{geometry}\relax
}
\RequirePackage{fix-cm}
\PassOptionsToPackage{numbers,sort&compress,square}{natbib}
\documentclass[smallcondensed,natbib]{svjour3}     %
\usepackage[utf8]{inputenc}
\def\makeheadbox{{%
\hbox to0pt{\vbox{\baselineskip=10dd\hrule\hbox
to\hsize{\vrule\kern3pt\vbox{\kern3pt
\hbox{
Kozdon, J.E., Erickson, B.A. \& Wilcox, L.C. Hybridized Summation-by-Parts
}
\hbox{
Finite Difference Methods. J Sci Comput 87, 85 (2021).
}
\hbox{
\url{https://doi.org/10.1007/s10915-021-01448-5}
}
\kern3pt}\hfil\kern3pt\vrule}\hrule}%
\hss}}}

\makeatletter
\if@twocolumn
  \renewcommand\normalsize{%
   \@setfontsize\normalsize\@xpt{12.5pt}%
   \abovedisplayskip=3 mm plus6pt minus 4pt
   \belowdisplayskip=3 mm plus6pt minus 4pt
   \abovedisplayshortskip=0.0 mm plus6pt
   \belowdisplayshortskip=2 mm plus4pt minus 4pt
   \let\@listi\@listI}%

  \renewcommand\small{%
   \@setfontsize\small{8.5pt}\@xpt
   \abovedisplayskip 8.5\p@ \@plus3\p@ \@minus4\p@
   \abovedisplayshortskip \z@ \@plus2\p@
   \belowdisplayshortskip 4\p@ \@plus2\p@ \@minus2\p@
   \def\@listi{\leftmargin\leftmargini
               \parsep 0\p@ \@plus1\p@ \@minus\p@
               \topsep 4\p@ \@plus2\p@ \@minus4\p@
               \itemsep0\p@}%
   \belowdisplayskip \abovedisplayskip}

\else
  \if@smallext
   \renewcommand\normalsize{%
   \@setfontsize\normalsize\@xpt\@xiipt
   \abovedisplayskip=3 mm plus6pt minus 4pt
   \belowdisplayskip=3 mm plus6pt minus 4pt
   \abovedisplayshortskip=0.0 mm plus6pt
   \belowdisplayshortskip=2 mm plus4pt minus 4pt
   \let\@listi\@listI}%

  \renewcommand\small{%
   \@setfontsize\small\@viiipt{9.5pt}%
   \abovedisplayskip 8.5\p@ \@plus3\p@ \@minus4\p@
   \abovedisplayshortskip \z@ \@plus2\p@
   \belowdisplayshortskip 4\p@ \@plus2\p@ \@minus2\p@
   \def\@listi{\leftmargin\leftmargini
               \parsep 0\p@ \@plus1\p@ \@minus\p@
               \topsep 4\p@ \@plus2\p@ \@minus4\p@
               \itemsep0\p@}%
   \belowdisplayskip \abovedisplayskip}
 \else
  \renewcommand\normalsize{%
   \@setfontsize\normalsize{9.5pt}{11.5pt}%
   \abovedisplayskip=3 mm plus6pt minus 4pt
   \belowdisplayskip=3 mm plus6pt minus 4pt
   \abovedisplayshortskip=0.0 mm plus6pt
   \belowdisplayshortskip=2 mm plus4pt minus 4pt
   \let\@listi\@listI}%

  \renewcommand\small{%
   \@setfontsize\small\@viiipt{9.25pt}%
   \abovedisplayskip 8.5\p@ \@plus3\p@ \@minus4\p@
   \abovedisplayshortskip \z@ \@plus2\p@
   \belowdisplayshortskip 4\p@ \@plus2\p@ \@minus2\p@
   \def\@listi{\leftmargin\leftmargini
               \parsep 0\p@ \@plus1\p@ \@minus\p@
               \topsep 4\p@ \@plus2\p@ \@minus4\p@
               \itemsep0\p@}%
   \belowdisplayskip \abovedisplayskip}
  \fi
\fi

\makeatother

\usepackage{pgfplots}
\pgfplotsset{compat=1.13}
\usepackage{subcaption}

\usepackage{mathtools}
\usepackage{stmaryrd}
\usepackage{amsfonts}
\newtheorem{assumption}{Assumption}
\usepackage[colorlinks=true]{hyperref}
\usepackage{booktabs}
\usepackage{authblk}
\smartqed  %

\usepackage{lineno}
\newcommand*\patchAmsMathEnvironmentForLineno[1]{%
  \expandafter\let\csname old#1\expandafter\endcsname\csname #1\endcsname
  \expandafter\let\csname oldend#1\expandafter\endcsname\csname end#1\endcsname
  \renewenvironment{#1}%
     {\linenomath\csname old#1\endcsname}%
     {\csname oldend#1\endcsname\endlinenomath}}%
\newcommand*\patchBothAmsMathEnvironmentsForLineno[1]{%
  \patchAmsMathEnvironmentForLineno{#1}%
  \patchAmsMathEnvironmentForLineno{#1*}}%
\AtBeginDocument{%
\patchBothAmsMathEnvironmentsForLineno{equation}%
\patchBothAmsMathEnvironmentsForLineno{align}%
\patchBothAmsMathEnvironmentsForLineno{flalign}%
\patchBothAmsMathEnvironmentsForLineno{alignat}%
\patchBothAmsMathEnvironmentsForLineno{gather}%
\patchBothAmsMathEnvironmentsForLineno{multline}%
}
\makeatletter
\newcommand{\footnoteref}[1]{%
\ltx@ifpackageloaded{hyperref}{%
  \ifHy@hyperfootnotes%
    \hbox{\hyperref[#1]{%
            \@textsuperscript {\normalfont \ref*{#1}}}}%
  \else%
    \hbox{\@textsuperscript {\normalfont \ref*{#1}}}%
  \fi%
}{%
    \hbox{\@textsuperscript {\normalfont \ref{#1}}}%
 }%
}
\makeatother

\SetMathAlphabet{\mathrm}{normal}{\encodingdefault}{cmss}{\mddefault}{n}
\SetMathAlphabet{\mathrm}{bold}{\encodingdefault}{cmss}{\bfdefault}{n}

\newcommand{\vv}[1]{{\boldsymbol{{ #1}} }}
\newcommand{\VV}[1]{{\boldsymbol{\tilde{ #1}} }}
\newcommand{\mm}[1]{{\boldsymbol{{ #1}} }}
\newcommand{\MM}[1]{{\boldsymbol{\tilde { #1}} }}
\newcommand{\eref}[1]{(\ref{#1})}
\newcommand{\sref}[1]{Section~\ref{#1}}
\newcommand{\aref}[1]{Section~\ref{#1}}
\newcommand{\pd}[2]{\frac{\partial #1}{\partial #2}}
\newcommand{\BB}{\mathcal{B}}
\newcommand{\FF}{\mathcal{F}}
\newcommand{\diag}{\text{diag}}
\newcommand{\avg}[1]{\ensuremath{\left\{\!\left\{ #1 \right\}\!\right\} } }
\newcommand{\jmp}[1]{\ensuremath{\left[\!\left[ #1 \right]\!\right]}}
\newcommand{\nullspace}{\ensuremath{\mbox{\rm null}}}
\newcommand{\spanspace}{\ensuremath{\mbox{\rm span}}}

\usepackage[draft,layout={inline,index}]{fixme}
\fxsetup{theme=color,mode=multiuser}
\FXRegisterAuthor{jek}{ajek}{\color{purple} JEK}
\FXRegisterAuthor{bae}{abae}{BAE}
\FXRegisterAuthor{lcw}{alcw}{\color{orange} LCW}

\title{Hybridized Summation-By-Parts Finite Difference Methods
\thanks{%
  J.E.K. was supported by National Science Foundation Award EAR-1547596.\\
  B.A.E. was supported by National Science Foundation Awards EAR-1547603 and
  EAR-1916992.\\
  The SCEC contribution number for this article is 10992.
  }
}

\begin{document}

\makeatletter
\let\ORIGINAL@spythm\@spythm{}
\def\@spythm#1#2#3#4[#5]{%
  \NR@gettitle{#5}%
  \ORIGINAL@spythm{#1}{#2}{#3}{#4}[#5]%
}
\makeatother

\author{Jeremy E. Kozdon \and
Brittany A. Erickson \and
Lucas C. Wilcox}

\institute{J. E. Kozdon and L. C. Wilcox \at%
          Department of Applied Mathematics,\\
          Naval Postgraduate School,\\
          833 Dyer Road,\\
          Monterey, CA 93943--5216\\
          \email{\{jekozdon,lwilcox\}@nps.edu}
          \and
          B. A. Erickson \at%
          Computer and Information Science\\
          1202 University of Oregon\\
          1477 E. 13th Ave.\\
          Eugene, OR 97403--1202\\
          \email{bae@cs.uoregon.edu}
          \and
          \begin{center}
            The views expressed in this document are those of the authors and do
            not reflect the official policy or position of the Department of
            Defense or the U.S. Government.\\
            Approved for public release; distribution unlimited
          \end{center}
          }

\date{published: June 2021}

\maketitle

\begin{abstract}
  We present a hybridization technique for summation-by-parts finite difference
  methods with weak enforcement of interface and boundary conditions for second
  order, linear elliptic partial differential equations. The method is based on
  techniques from the hybridized discontinuous Galerkin literature where local
  and global problems are defined for the volume and trace grid points,
  respectively. By using a Schur complement technique the volume points can be
  eliminated, which drastically reduces the system size. We derive both the
  local and global problems, and show that the resulting linear systems are
  symmetric positive definite. The theoretical stability results are confirmed
  with numerical experiments as is the accuracy of the method.
\end{abstract}

\section{Introduction}
High-order finite difference methods have a long and rich history for solving
second order, elliptic partial differential equations (PDEs); see for instance
the short historical review of \citet{Thomee2001}.
When complex geometries are involved, finite difference methods are similar to
finite element methods in that unstructured meshes and coordinate transforms
can be used to handle complex geometries \citep{NordstromCarpenter2001}.
Summation-by-parts (SBP) finite difference methods \citep{KS74,KS77, Strand94,
MN04, Mat12} have been particularly effective for such problems, since
inter-block coupling conditions be can be handled weakly using the simultaneous
approximation term (SAT) method \citep{CarpenterGottliebAbarbanel1994,
CarpenterNordstromGottlieb1999}.

The combined SBP-SAT approach has been used extensively for problems that arise
in the natural sciences where physical interfaces are ubiquitous, for example in
earthquake problems where faults separate continental and oceanic crustal blocks
or in multiphase fluids with discontinuous properties
\citep{Kozdon2012InteractionOW, EricksonDay2016, Karlstrom2016, Lotto2015}.
The present work is particularly motivated by models of earthquake nucleation
and rupture propagation over many thousands of years, where the slow, quiescent
periods between earthquakes represent quasi-steady state problems
\citep{Erickson2014}.
In the steady-state regime, an elliptic PDE must be
repeatedly solved, which results in large linear systems of equations for
complex problems.

In this work we propose a hybridization technique for SBP-SAT methods in order
to reduce the size of the linear systems.
The inspiration for this is static condensation and hybridization for finite
element methods \citep{Guyan1965, CockburnGopalakrishnanLazarov2009}.
These techniques reduce system size by writing the numerical method in a way
that allows the Schur complement to be used to eliminate degrees of freedom from
within the element leaving only degrees of freedom on element boundaries.
SBP-SAT methods have a similar discrete structure to discontinuous Galerkin
methods, with the penalty terms in SBP-SAT methods being analogous to the
numerical fluxes in discontinuous Galerkin methods.

Here we introduce independent trace variables along the faces of the blocks, and
the inter-block coupling penalty terms are only a function of the trace
variables.
Thus, the solution in each block is uniquely determined by the trace variables
which are applied as Dirichlet boundary data.
The problem is broken into two pieces, a \emph{local problem},
which is the solution within the block given the trace data, and a
\emph{global problem}, which is the value of the trace variable given the block
data.
Using a Schur complement technique either set of variables can be eliminated.
When the trace variables are eliminated the scheme is similar to existing
SBP-SAT schemes, for instance the method of \citet{VM14}.
If on the other hand the volume variables are eliminated and the trace variables
are retained, the system size is drastically reduced since the system only involves
the unknowns along the block faces.
That said, the cost of forming this later Schur complement system arises from
the need to invert each finite difference block (though we note that each
inverse is independent, involving only the block local degrees of freedom).

The developed method is symmetric positive definite for the monolithic system
(trace and volume variables) as are the two Schur complement systems.
Thus, the elliptic discretization is stable.
Importantly, these properties are shown to hold even if the elliptic problem is
variable coefficient or involves curvilinear blocks.
Since the  discretization is based on the hybridized interior penalty method
\citep[IP-H]{CockburnGopalakrishnanLazarov2009}, there is a (spatially varying)
penalty parameter that must be sufficiently large for stability and a bound for
this penalty is given.
It is also shown that the penalty parameter can be determined purely from the
local problem, independent of the neighboring blocks.

The paper is organized as follows: In \sref{sec:sbp} we detail the block
decomposition and SBP operators.
\sref{sec:model} describes the model problem, an elliptic PDE, along with
boundary and interface conditions which allow for jump discontinuities and
material contrasts.
\sref{sec:hsbp} details the hybridized scheme, including the local and global
problems.
Proofs of positive-definiteness of both systems are provided; these results are
confirmed with numerical experiments in \sref{sec:numerical}.
\sref{sec:numerical} also provides results from convergence tests using an exact
solution, and we conclude with a summary in \sref{sec:conclusions}.

\section{Domain decomposition and SBP operators}\label{sec:sbp}
As noted above, we apply the class of high-order accurate SBP finite difference
methods which were introduced for first derivatives in \citet{KS74, KS77,
Strand94}, and for second derivatives by \citet{MN04}, with the variable
coefficients treated in \citet{Mat12}. In addition to high-order accuracy, SBP
methods can be combined with various boundary treatments so that the resulting
linear PDE discretization is provably stable. In \sref{sec:hsbp} we use weak
enforcement of boundary and interface conditions with the
Simultaneous-Approximation-Term (SAT) method. Here we introduce notation related
to the decomposition of the computational domain into blocks as well as
one-dimensional and two-dimensional SBP operators for first and second
derivatives.

\subsection{Domain Decomposition}
We let the computational domain be $\Omega \subset \mathbb{R}^{2}$ which is
partitioned into $N_b$ non-overlapping curved quadrilateral blocks; the
partitioning is denoted $\BB(\Omega)$. For each block $B \in \BB(\Omega)$ we
assume that there exists a diffeomorphic mapping from the reference block
$\hat{B} = [0,1] \times [0,1]$ to $B$. The mapping $\left(x^{B}(r,s),
y^{B}(r,s)\right)$ goes from the reference block to the physical block and
$\left(r^{B}(x,y), s^{B}(x,y)\right)$ is the inverse mapping.  An example of
this is shown in Figure~\ref{fig:coord:trans}; the figure also shows the face
numbering for the reference block.
\begin{figure}
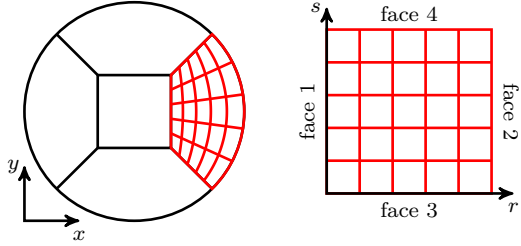

  \begin{center}

     \caption{(left) Block decomposition of a disk with a single curved block
    highlighted along with its grid lines in physical space.
    (right) Mapping of the highlighted block to the reference domain; shown in
    the figure is the convention used to number the faces of the reference
    block.\label{fig:coord:trans}}
  \end{center}
\end{figure}

As will be seen in \sref{sec:model}, the transformation to the reference block
requires metric relations that relate the physical and reference derivatives.
Four relations that are particularly useful are
\begin{align*}
  J\pd{r}{x} &=  \pd{y}{s}, &
  J\pd{s}{y} &=  \pd{x}{r}, &
  J\pd{s}{x} &= -\pd{y}{r}, &
  J\pd{r}{y} &= -\pd{x}{s},
\end{align*}
with $J$ being the Jacobian determinant for block $B$,
\begin{equation*}
  J =
  \pd{x}{r}\pd{y}{s} - \pd{x}{s}\pd{y}{r};
\end{equation*}
for simplicity of notation, unless required we suppress the block $B$
superscript and the relations should be understood as applying to a single
block.  For face $k$ of a block, the surface Jacobian is
\begin{equation*}
  \mathcal{S}_{J,k} =
  \begin{cases}
    \sqrt{
      {\left(\pd{x}{s}\right)}^{2}
      +
      {\left(\pd{y}{s}\right)}^{2}
    },& \mbox{ if } k = 1,2,\\
    \sqrt{
      {\left(\pd{x}{r}\right)}^{2}
      +
      {\left(\pd{y}{r}\right)}^{2}
    },& \mbox{ if } k = 3,4,
  \end{cases}
\end{equation*}
and the outward unit normal vectors are
\begin{equation*}
  \begin{alignedat}{2}
    \mathcal{S}_{J,1}
    \vv{\hat{n}}_{1}
    &=
    \begin{bmatrix}
      - \pd{y}{s}\\
      \phantom{-}\pd{x}{s}
    \end{bmatrix},~&
    \mathcal{S}_{J,2}
    \vv{\hat{n}}_{2}
    &=
    \begin{bmatrix}
      \phantom{-}\pd{y}{s}\\
      - \pd{x}{s}
    \end{bmatrix},\\
    \mathcal{S}_{J,3}
    \vv{\hat{n}}_{3}
    &=
    \begin{bmatrix}
      \phantom{-}\pd{y}{r}\\
      - \pd{x}{r}
    \end{bmatrix},&
    \mathcal{S}_{J,4}
    \vv{\hat{n}}_{4}
    &=
    \begin{bmatrix}
      - \pd{y}{r}\\
      \phantom{-}\pd{x}{r}
    \end{bmatrix}.
  \end{alignedat}
\end{equation*}
\subsection{One-Dimensional SBP operators}
Let the domain $0 \leq r \leq 1$ be discretized with $N+1$ evenly spaced grid
points $r_i = i \, h$, $i = 0, \dots, N$ and spacing $h = 1/N$. The projection
of a function $u$ onto the computational grid is taken
to be $\vv{u} = {[u_{0},\,u_{1},\,\dots,\,u_{N}]}^{T}$; if $u$ is known then
$\vv{u}$ is often taken to be the interpolant at the grid points. The grid basis
vector $\vv{e}_{j}$ is $1$ at grid point $j$ and zero at all other grid points
and $u_{j} = \vv{e}_{j}^{T}\vv{u}$.

\begin{definition}[First Derivative]
  A matrix $\mm{D}_{r}$ is a called an SBP approximation to $\partial u/\partial
  r$ if it can be decomposed as $\mm{H}\mm{D}_{r} = \mm{Q}$ with $\mm{H}$ being
  symmetric positive definite and $\mm{Q}$ being such that $\vv{u}^{T}(\mm{Q} +
  \mm{Q}^{T})\vv{v} = u_{N}v_{N} - u_{0}v_{0}$.
\end{definition}
In this work we only consider diagonal-norm SBP, i.e., finite difference
operators where $\mm{H}$ is a diagonal matrix and $\mm{D}_{r}$ is the standard
central finite difference matrix in the interior which transitions to one-sided
at the boundaries. The condition on $\mm{Q}$ can also be written as $\mm{Q} +
\mm{Q}^{T} = \vv{e}_{N}\vv{e}_{N}^{T} - \vv{e}_{0}\vv{e}_{0}^{T}$.

The operator $\mm{D}_{r}$ is called SBP because the integration-by-parts
property
\begin{equation*}
  \int_0^1 u \pd{v}{r} + \int_0^1 \pd{u}{r} v = uv\bigg|_0^1,
\end{equation*}
is mimicked discretely by
\begin{equation*}
  \vv{u}^T \mm{H} \mm{D}_{r}\vv{v} + \vv{u}^T \mm{D}_{r}^{T} \mm{H}\vv{v}
  =
  \vv{u}^T\left( \mm{Q} + \mm{Q}^T\right)\vv{v}
  =
  u_{N}v_{N} - u_{0}v_{0}.
\end{equation*}

\begin{definition}[Second Derivative]
  A matrix $\mm{D}_{rr}^{(c)}$ is a called an SBP approximation to
  $\frac{\partial}{\partial r}\left(c\frac{\partial u}{\partial r}\right)$ if it
  can be decomposed as $\mm{H}\mm{D}_{rr}^{(c)} = -\mm{A}^{(c)} +
  c_{N}\vv{e}_{N}\vv{d}_{N}^{T} - c_{0}\vv{e}_{0}\vv{d}_{0}^{T}$ where
  $\mm{A}^{(c)}$ is symmetric positive definite and $\vv{d}_{0}^{T}\vv{u}$ and
  $\vv{d}_{N}^{T}\vv{u}$ are approximations of the first derivative of $u$ at
  the boundaries.
\end{definition}

The operator $\mm{D}^{(c)}_{rr}$ is called SBP because the
integration-by-parts equality
\begin{equation*}
  \int_0^1 u \pd{}{r}\left(c\pd{v}{r}\right)
 + \int_0^1 \pd{u}{r} c\pd{v}{r}
  =
    u c \pd{v}{r} \bigg|_0^1,
\end{equation*}
is mimicked discretely by
\begin{equation*}
  \vv{u}^T \mm{H} \mm{D}^{(c)}_{rr}\vv{v}
  +\vv{u}^T\mm{A}^{(c)}\vv{v}
  =
    c_{N}u_{N} \vv{d}_{N}^{T}\vv{v}
  - c_{0}u_{0} \vv{d}_{0}^{T}\vv{v}.
\end{equation*}

\begin{definition}[Compatability]
  Matrices $\mm{D}_{r}$ and $\mm{D}_{rr}^{(c)}$ are called compatible SBP
  operators if they use the same matrix $\mm{H}$ and the remainder matrix
  $\mm{R}^{(c)} = \mm{A}^{(c)} - \mm{D}_{r}^{T} \mm{C} \mm{H} \mm{D}_{r}$ is
  symmetric positive definite with $\mm{C} = \diag(\vv{c})$ being a diagonal
  matrix constructed from the grid interpolant of $c$.
\end{definition}

It is important to note that compatibility does not assume that $\vv{d}_{0}^{T}$
and $\vv{d}_{N}^{T}$ are the first and last rows of $\mm{D}_{r}$. When this is
the case the operators are called fully-compatible
\citep{MattssonParisi2010CICP} and such operators are not used in this work.

As noted above, we only consider diagonal-norm SBP finite difference operators.
In the interior the operators use the minimum bandwidth central difference
stencil and transition to one-sided near the boundary in a manner that maintains
the SBP property.  If the interior operator has accuracy $2p$, then the interior
stencil bandwidth is $2p+1$ and the boundary operator has accuracy $p$. The
first and second derivative operators used are those given in
\citet{Strand94}\footnote{\label{footnote:x1}The free parameter in the $2p=6$
operator from \citet{Strand94} is taken to be $x_1=0.70127127127127$.  This
choice of free parameter is necessary for the values of the Borrowing Lemma
given in \citet{VM14} to hold; the Borrowing Lemma is discussed in
\aref{sec:app:loc:PD}.} and \citep{Mat12}, respectively.  In
\sref{sec:numerical} we will use operators with interior accuracy $2p = 2$, $4$,
and $6$. The expected global order of accuracy is the minimum of $2p$ and $p+2$
as evidenced experimentally \citep{Mattsson2009, VM14} and proved rigorously for
the Schr\"{o}dinger equation \citep{Nissen2013}.  In \sref{sec:numerical} we
verify this result for the hybridized scheme through convergence tests.

\begin{remark}
  If the second derivative finite difference operator is defined by repeated
  applications of the first derivatives operator, e.g, $\mm{D}_{rr}^{(c)} =
  \mm{D}_{r} \mm{C} \mm{D}_{r}$, then the operator is fully compatible with
  $\mm{R}^{(c)}$ being the zero matrix but the operator does not have minimal
  bandwidth.
\end{remark}
\subsection{Two-Dimensional SBP operators}\label{sec:2d}
Two-dimensional SBP operators can be developed for rectangular domains by
applying the one-dimensional operators in a tensor product fashion (i.e.,
dimension-by-dimension application of the one-dimensional operators). Here we
describe the operators for the reference block $\hat{B} = [0, 1] \times
[0, 1]$. We assume that the domain is discretized using an $(N+1) \times (N+1)$
grid of points where grid point $(i,j)$ is at $(r_i, s_j) = (ih, jh)$ for $0
\leq i, j \leq N$ with $h = 1/N$; the generalization to different numbers of
grid points in each dimension complicates the notation but does not impact the
construction of the method and is discussed later.

A 2D grid function $\VV{u}$ is taken to be a stacked vector of vectors with
$\VV{u} = {[\vv{u}_0^T,\, \vv{u}_1^T, \, \dots, \, \vv{u}_{N}^T]}^T$ and
$\vv{u}_i^T = {[u^{0i}, \, u^{1i},\, \dots\, , u^{Ni}]}^{T}$ where $u^{ji}
\approx u(r_j,s_i)$.

Derivative approximations are taken to be of the form
\begin{equation}
  \label{eqn:deriv:approx:form}
  \begin{alignedat}{2}
    \pd{}{r}\left(c_{rr}\pd{u}{r}\right) &\approx \MM{{D}}_{rr}^{(c_{rr})}\VV{u},~&
    \pd{}{s}\left(c_{ss}\pd{u}{s}\right) &\approx \MM{{D}}_{ss}^{(c_{ss})}\VV{u},\\
    \pd{}{r}\left(c_{rs}\pd{u}{s}\right) &\approx \MM{{D}}_{rs}^{(c_{rs})}\VV{u},~&
    \pd{}{s}\left(c_{sr}\pd{u}{s}\right) &\approx \MM{{D}}_{sr}^{(c_{sr})}\VV{u}.
  \end{alignedat}
\end{equation}
To explicitly define the derivative operators, we first let $\VV{c}_{rr}$ be the
grid interpolant of the weighting function $c_{rr}$ and define $\MM{C}_{rr} =
\diag(\VV{c}_{rr})$. Additionally, the diagonal matrices of the coefficient
vectors along each of the grid lines are
\begin{align*}
  \mm{C}_{rr}^{:j} &= \diag\left(c_{rr}^{0j},\,\dots,\,c_{rr}^{Nj}\right),&
  \mm{C}_{rr}^{i:} &= \diag\left(c_{rr}^{i0},\,\dots,\,c_{rr}^{iN}\right).
\end{align*}
Similar matrices are constructed for $c_{ss}$, $c_{rs}$, and $c_{sr}$.  With
this, the derivative operators in \eref{eqn:deriv:approx:form} are
\begin{subequations}
  \begin{align*}
    {\left(\mm{H} \otimes \mm{H}\right)}
    \MM{{D}}^{(c_{rr})}_{rr} &=
    -\MM{A}^{(c_{rr})}_{rr}
    + \left(\mm{H} \mm{C}_{rr}^{N:} \otimes \vv{e}_{N}\vv{d}_{N}^{T}\right)
    - \left(\mm{H} \mm{C}_{rr}^{0:} \otimes \vv{e}_{0}\vv{d}_{0}^{T}\right)
    ,\\
    {\left(\mm{H} \otimes \mm{H}\right)}
    \MM{{D}}^{(c_{ss})}_{ss} &=
    -\MM{A}^{(c_{ss})}_{ss}
    + \left(\vv{e}_{N}\vv{d}_{N}^{T} \otimes \mm{H} \mm{C}_{ss}^{:N}\right)
    - \left(\vv{e}_{0}\vv{d}_{0}^{T} \otimes \mm{H} \mm{C}_{ss}^{:0}\right)
    ,\\
    \notag
    {\left(\mm{H} \otimes \mm{H}\right)}
    \MM{{D}}^{(c_{rs})}_{rs} &=
    \left(\mm{I} \otimes \mm{Q}\right)
    \MM{C}_{rs}
    \left(\mm{Q} \otimes \mm{I}\right)\\
    &=
    -\MM{A}^{(c_{rs})}_{rs}
    + \left(\mm{C}_{rs}^{N:} \mm{Q} \otimes \vv{e}_{N}\vv{e}_{N}^{T}\right)
    - \left(\mm{C}_{rs}^{0:} \mm{Q} \otimes \vv{e}_{0}\vv{e}_{0}^{T}\right)
    ,\\
    \notag
    {\left(\mm{H} \otimes \mm{H}\right)}
    \MM{{D}}_{sr}^{(c_{sr})} &=
    \left(\mm{Q} \otimes \mm{I}\right)
    \MM{C}_{sr}
    \left(\mm{I} \otimes \mm{Q}\right)\\
    &=
    -\MM{A}^{(c_{sr})}_{sr}
    + \left(\vv{e}_{N}\vv{e}_{N}^{T} \otimes \mm{C}_{sr}^{:N} \mm{Q}\right)
    - \left(\vv{e}_{0}\vv{e}_{0}^{T} \otimes \mm{C}_{sr}^{:0} \mm{Q}\right)
    ,
  \end{align*}
\end{subequations}
where $\otimes$ denotes the Kronecker product of two matrices.
Here, the matrices $\MM{A}_{rr}^{(c_{rr})}$, $\MM{A}_{ss}^{(c_{ss})}$,
$\MM{A}_{rs}^{(c_{rs})}$, and $\MM{A}_{sr}^{(c_{sr})}$ are
\begin{equation}
  \label{eqn:A}
  \begin{alignedat}{2}
    \MM{A}_{rr}^{(c_{rr})} &=
    \left(\mm{H} \otimes \mm{I}\right)
    \left[
      \sum_{j=0}^{N}
      \left(\vv{e}_{j} \otimes \mm{I}\right)
      \mm{A}^{\left(C^{:j}_{rr}\right)}
      \left(\vv{e}_{j}^{T} \otimes \mm{I}\right)\right],\\
    \MM{A}_{ss}^{(c_{ss})} &=
    \left(\mm{I} \otimes \mm{H}\right)
    \left[
      \sum_{i=0}^{N}
      \left(\mm{I} \otimes \vv{e}_{i}\right)
      \mm{A}^{\left(C^{i:}_{ss}\right)}
      \left(\mm{I} \otimes \vv{e}_{i}^{T}\right)\right],\\
    \MM{A}_{rs}^{(c_{rs})} &=
    \left(\mm{I} \otimes \mm{Q}^{T}\right)
    \MM{C}_{rs}
    \left(\mm{Q} \otimes \mm{I}\right)
    ,\\
    \MM{A}_{sr}^{(c_{sr})} &=
    \left(\mm{Q}^{T} \otimes \mm{I}\right)
    \MM{C}_{sr}
    \left(\mm{I} \otimes \mm{Q}\right)
    ,
  \end{alignedat}
\end{equation}
and can be viewed as approximations of the following integrals:
\begin{equation*}
  \begin{alignedat}{2}
  \int_{\hat{B}} \pd{u}{r} c_{rr} \pd{v}{r} &\approx \VV{u}^{T} \MM{A}_{rr}^{(c_{rr})} \VV{v},\quad&
  \int_{\hat{B}} \pd{u}{r} c_{rs} \pd{v}{s} &\approx \VV{u}^{T} \MM{A}_{rs}^{(c_{rs})} \VV{v},\\
  \int_{\hat{B}} \pd{u}{s} c_{sr} \pd{v}{r} &\approx \VV{u}^{T} \MM{A}_{sr}^{(c_{sr})} \VV{v},\quad&
  \int_{\hat{B}} \pd{u}{s} c_{ss} \pd{v}{s} &\approx \VV{u}^{T} \MM{A}_{ss}^{(c_{ss})} \VV{v}.
  \end{alignedat}
\end{equation*}
The following equality will be useful later which splits the volume and surface
contributions:
\begin{align}
  \label{eqn:A:structure}
  \notag
  {\left(\mm{H} \otimes \mm{H}\right)}&
  \left[
    -\MM{{D}}^{(c_{rr})}_{rr} - \MM{{D}}^{(c_{rs})}_{rs}
    -\MM{{D}}^{(c_{sr})}_{sr} - \MM{{D}}^{(c_{ss})}_{ss}
  \right]\\
  =\;&
    \MM{{A}}^{(c_{rr})}_{rr} + \MM{{A}}^{(c_{rs})}_{rs}
  + \MM{{A}}^{(c_{sr})}_{sr} + \MM{{A}}^{(c_{ss})}_{ss}\\
  \notag
  &
  - \mm{L}_{1}^{T} \mm{G}_{1}
  - \mm{L}_{2}^{T} \mm{G}_{2}
  - \mm{L}_{3}^{T} \mm{G}_{3}
  - \mm{L}_{4}^{T} \mm{G}_{4}.
\end{align}
Here the face point extraction operators are defined as
\begin{align*}
  \mm{L}_{1} &= \mm{I} \otimes \mm{e}_{0}^{T}, &
  \mm{L}_{2} &= \mm{I} \otimes \mm{e}_{N}^{T}, &
  \mm{L}_{3} &= \mm{e}_{0}^{T} \otimes \mm{I}, &
  \mm{L}_{4} &= \mm{e}_{N}^{T} \otimes \mm{I},
\end{align*}
and the matrices which compute the weighted boundary derivatives are
\begin{equation}\label{eqn:weight:bnd:der}
  \begin{split}
    \mm{G}_{1} =&
    -
    \left(\mm{H}\mm{C}_{rr}^{0:} \otimes \mm{d}_{0}^{T}\right)
    -
    \Big(\mm{C}_{rs}^{0:} \mm{Q} \otimes \mm{e}_{0}^{T}\Big),\\
    \mm{G}_{2} =&
    \phantom{-}
    \left(\mm{H}\mm{C}_{rr}^{N:} \otimes \mm{d}_{N}^{T}\right)
    +
    \left(\mm{C}_{rs}^{N:} \mm{Q} \otimes \mm{e}_{N}^{T}\right),\\
    \mm{G}_{3} =&
    -
    \left(\mm{d}_{0}^{T}  \otimes  \mm{H}\mm{C}_{ss}^{:0}\right)
    -
    \Big(\mm{e}_{0}^{T} \otimes \mm{C}_{sr}^{:0} \mm{Q}\Big) ,\\
    \mm{G}_{4} =&
    \phantom{-}
    \left(\mm{d}_{N}^{T} \otimes  \mm{H}\mm{C}_{ss}^{:N}\right)
    +
    \left(\mm{e}_{N}^{T} \otimes \mm{C}_{sr}^{:N} \mm{Q}\right).
  \end{split}
\end{equation}
The matrix $\mm{G}_{f}$ should be thought of as approximating the integral of
the boundary derivative, for example
\begin{equation*}
  \vv{v}^{T}\mm{L}_{1}^{T}\mm{G}_{1}\vv{u}
  \approx
  -\int_{0}^{1} {\left.\left(v\left(c_{rr}\pd{u}{r} +
  c_{rs}\pd{u}{s}\right)\right)\right|}_{r=1}.
\end{equation*}

\begin{remark}
  As noted above, for simplicity of notation we have assumed that the grid
  dimension is the same in both directions. This can be relaxed by letting the
  first argument in the Kronecker products be with respect to the $s$-direction
  and the second with respect to the $r$-direction. If the grid were different
  in each direction then, for example, $\left(\mm{H} \otimes \mm{H}\right)$
  would be replaced by $\left(\mm{H}_{s} \otimes \mm{H}_{r}\right)$ where
  $\mm{H}_{r}$ and $\mm{H}_{s}$ are the one-dimensional SBP norm matrices based
  on grids of size $N_{r}+1$ and $N_{s}+1$, respectively.
\end{remark}

\section{Model Problem}\label{sec:model}
As a model problem we consider the following scalar, anisotropic elliptic
equation in two spatial dimensions for the field $u$:
\begin{subequations}\label{eqn:gov}
  \begin{alignat}{2}
    \label{eqn:gov:nab}
    &-\nabla\cdot\left(\mm{b} \nabla u\right) = f,
    &&\mbox{ on } \Omega,\\
    \label{eqn:gov:dir}
    &u = g_{D},
    &&\mbox{ on } \partial \Omega_{D},\\
    \label{eqn:gov:neu}
    &\vv{n} \cdot \mm{b} \nabla u = g_{N},
    &&\mbox{ on } \partial \Omega_{N},\\
    \label{eqn:gov:inter}
    &
    \begin{cases}
      \avg{\vv{n} \cdot \mm{b} \nabla u} = 0,\\
      \jmp{u} = \delta,
    \end{cases}
    &&\mbox{ on } \Gamma_{I}.
  \end{alignat}
\end{subequations}
Here $\mm{b}(x, y)$ is a matrix valued function that is symmetric positive
definite and the scalar function $f(x,y)$ is a source function.  The boundary of
the domain has been partitioned into Dirichlet and Neumann segments, i.e.,
$\partial \Omega = \partial \Omega_{D} \cup \partial \Omega_{N}$ and $\partial
\Omega_{D} \cap \partial \Omega_{N} = \emptyset$. In the Neumann boundary
conditions, the vector $\vv{n}$ is the outward pointing normal. The functions
$g_{D}$ and $g_{N}$ are given data at the boundaries. An internal interface
$\Gamma_{I}$ has also been introduced. Along this interface the
$\mm{b}$-weighted normal derivative is taken to be continuous,  with jumps
allowed in the scalar field $u$; this allowance is made so that the scheme can
be used for the earthquake problems that motivate the work. Here $\avg{w}
= w^{+} + w^{-}$ denotes the sum of the scalar quantity on both sides of the
interface and $\jmp{w} = w^{+} - w^{-}$ is the difference across the interface;
the side defined as the plus- and minus-side are arbitrary though the choice
affects the sign of the jump data $\delta$.

Governing equations \eref{eqn:gov} are not solved directly on $\Omega$. Instead,
the equations are solved over each $B \in \BB(\Omega)$, where along each edge
of $B$ either continuity of the solution and the $\mm{b}$-weighted normal
derivative are enforced, or the appropriate boundary (or interface) condition.
Additionally, we do not solve directly on $B$ but instead transform to the
reference block $\hat{B}$. With this, \eref{eqn:gov:nab} becomes for each $B \in
\BB$:
\begin{subequations}\label{eqn:gov:transformed}
\begin{equation}\label{eqn:gov:Bhat}
  -\hat{\nabla} \cdot \left(\mm{c} \hat{\nabla} u\right) = Jf,
\end{equation}
where $\hat{\nabla}u = {[\pd{u}{r},\,\pd{u}{s}]}^{T}$, i.e., the $\hat{\nabla}$
is the del operator with respect to $(r,s)$, and the matrix valued
coefficient function $\mm{c}(r,s)$ has entries
\begin{align}
    c_{rr} &= J\left(b_{xx}\pd{r}{x}\pd{r}{x} + 2b_{xy}\pd{r}{x}\pd{r}{y}+
    b_{yy}\pd{r}{y}\pd{r}{y}\right),\\
    c_{ss} &= J\left(b_{xx}\pd{s}{x}\pd{s}{x} + 2b_{xy}\pd{s}{x}\pd{s}{y}+
    b_{yy}\pd{s}{y}\pd{s}{y}\right),\\
    c_{rs} &= c_{sr} = J\left(b_{xx}\pd{r}{x}\pd{s}{x} +
    b_{xy}\left(\pd{r}{x}\pd{s}{y}+\pd{r}{y}\pd{s}{x}\right)+
    b_{yy}\pd{r}{y}\pd{s}{y}\right),
\end{align}
where $b_{xx}$, $b_{yy}$, and $b_{xy} = b_{yx}$ are the four components of
$\mm{b}$.  For simplicity of notation we have suppressed the subscript $B$ on
terms in~\eref{eqn:gov:transformed} and following. If $J > 0$ then the matrix
formed by $c_{rr}$, $c_{ss}$, and $c_{rs}=c_{sr}$ is symmetric positive definite
and \eref{eqn:gov:Bhat} is of the same form as \eref{eqn:gov:nab} except on the
unit square domain $\hat{B}$.

The boundary conditions \eqref{eqn:gov:dir}--\eqref{eqn:gov:neu} and interface
conditions \eqref{eqn:gov:inter} are similarly transformed.  Namely, letting
$\partial \hat{B}_{k}$ for $k = 1,2,3,4$ be the faces of $\hat{B}$, we then
require that for each $k$:
  \begin{alignat}{2}
    &u = g_{D},
    &&\mbox{ if } \hat{B}_{k} \cap \partial \Omega_{D} \ne \emptyset,\\
    &\vv{\hat{n}}_{k} \cdot \mm{c} \hat{\nabla}u
    = \mathcal{S}_{J,k} g_{N},
    &&\mbox{ if } \hat{B}_{k} \cap \partial \Omega_{N} \ne \emptyset,\\
    \label{eqn:bc:jump}
    &
    \begin{cases}
      \avg{\vv{\hat{n}}_{k} \cdot \mm{c} \nabla u} = 0,\\
      \jmp{u} = \delta,
    \end{cases}
    &&\mbox{ if } \hat{B}_{k} \cap \Gamma_{I} \ne \emptyset,\\
    \label{eqn:bc:locked}
    &
    \begin{cases}
      \avg{\vv{\hat{n}}_{k} \cdot \mm{c} \nabla u} = 0,\\
      \jmp{u} = 0,
    \end{cases}
    &&\mbox{ otherwise}.
  \end{alignat}
\end{subequations}
Here $\vv{\hat{n}}_{k}$ is the outward pointing normal to face $\partial
\hat{B}_{k}$ in the reference space (not the physical space) and
$\mathcal{S}_{J,k}$ is the surface Jacobian which arises due to the fact that
$\mm{c}$ includes metric terms.  Condition \eref{eqn:bc:locked} is the same as
\eref{eqn:bc:jump} if $\delta$ is defined to be $0$ on these faces.
\section{Hybridized SBP Scheme}\label{sec:hsbp}
In the finite element literature, a hybrid method has one unknown
function on element interiors and a second unknown function on element traces
\citep[page 421]{Ciarlet2002FEM}.  For SBP methods, the
big idea is to write the method in terms of local problems and a global problem.
In the local problems, for each $B \in \BB$ the trace of the solution (i.e., the
boundary and interface data) is assumed and the transformed
equation~\eref{eqn:gov:transformed} is solved locally over $B$.
In the global problem the solution traces for each $B\in\BB$ are coupled.  As
will be shown, this technique will result in a linear system of the form
\begin{align}\label{eqn:full:system}
  \begin{bmatrix}
    \mm{\bar{M}} & \mm{\bar{F}}\\
    \mm{\bar{F}}^{T} & \mm{\bar{D}}
  \end{bmatrix}
  \begin{bmatrix}
    \vv{\bar{u}}\\
    \vv{\bar{\lambda}}
  \end{bmatrix}
  =
  \begin{bmatrix}
    \vv{\bar{g}}\\
    \vv{\bar{g}}_{\delta}
  \end{bmatrix}.
\end{align}
Here $\vv{\bar{u}}$ is the approximate solution to \eref{eqn:gov:transformed} at
all the grid points and $\vv{\bar{\lambda}}$ are the trace variables along
internal interfaces; trace variables related to boundary conditions can be
eliminated. The matrix $\mm{\bar{M}}$ is block diagonal with one symmetric
positive definite block for each $B \in \BB$, $\mm{\bar{D}}$ is diagonal, and
the matrix $\mm{\bar{F}}$ is sparse and incorporates the coupling conditions.
The right-hand side vector $\vv{\bar{g}}$ incorporates the boundary data
($g_{D}$, $g_{N}$) and source terms whereas $\vv{\bar{g}}_{\delta}$ incorporates
the interface data $\delta$.

Using the Schur complement we can transform \eref{eqn:full:system} to
\begin{align}\label{eqn:schur:system}
  \left(\mm{\bar{D}} - \mm{\bar{F}}^{T} \mm{\bar{M}}^{-1} \mm{\bar{F}}\right)
  \vv{\bar{\lambda}} = \vv{\bar{g}}_{\delta}-\mm{\bar{F}}^{T} \mm{\bar{M}}^{-1} \vv{\bar{g}},
\end{align}
resulting in a substantially reduced problem size since the number of trace
variables is significantly smaller than the number of solution variables.
Since $\mm{\bar{M}}$ is block diagonal, the inverse can be applied in a
decoupled manner for each $B \in \BB$. Thus there is a trade-off between the
number of blocks and the size of system \eref{eqn:schur:system}, since for a
fixed resolution increasing the number of blocks means that $\mm{\bar{M}}$ will
be more efficiently factored but the size of \eref{eqn:schur:system} will
increase through the introduction of additional trace variables.

Now that the big picture is laid, we proceed to introduce the local problem
(thus defining $\mm{\bar{M}}$) and then the global coupling (which defines
$\mm{\bar{F}}$ and $\mm{\bar{D}}$).

\subsection{The Local Problems}
For each $B \in \BB$ we solve \eref{eqn:gov:Bhat} with boundary conditions
\begin{align}
  \label{eqn:bc:lambda}
  u = \lambda_{k} \mbox{ on } \partial \hat{B}_{k} \mbox{ for } k =
  1,\,2,\,3,\,4,
\end{align}
where for now we assume that the trace functions $\lambda_{k}$ are known; later
these will be defined in terms of the boundary and coupling conditions. Using
the SBP operators defined in \sref{sec:2d} a discretization of
\eref{eqn:gov:Bhat} is
\begin{align}
  \label{eqn:disc}
  -\MM{{D}}_{rr}^{(C_{rr})} \VV{u}
  -\MM{{D}}_{rs}^{(C_{rs})} \VV{u}
  -\MM{{D}}_{sr}^{(C_{sr})} \VV{u}
  -\MM{{D}}_{ss}^{(C_{ss})} \VV{u} =
  \MM{J} \VV{f}
  + \VV{{b}}_{1}
  + \VV{{b}}_{2}
  + \VV{{b}}_{3}
  + \VV{{b}}_{4}.
\end{align}
Here $\VV{u}$ is the vector solution and $\MM{J}\VV{f}$ is the grid
approximation of $Jf$. The terms $\VV{{b}}_{1}$, $\VV{{b}}_{2}$, $\VV{{b}}_{3}$,
and $\VV{{b}}_{4}$ are the penalty terms which incorporate local boundary
conditions \eref{eqn:bc:lambda}; this is known as the SAT method and is
equivalent to the numerical flux in discontinuous Galerkin
formulations~\citep{CarpenterGottliebAbarbanel1994, Gassner2013}. The penalty
terms are taken to be of the form
\begin{equation*}
  \begin{split}
    \left(\mm{H} \otimes \mm{H}\right) \VV{{b}}_{1} &=
    \mm{G}_{1}^{T}
    \left[
      \mm{L}_{1} \VV{u}
      -
      \vv{\lambda}_{1}
      \right]
    +
    \mm{L}_{1}^{T}
    \left[
      \mm{H}\vv{\hat{\sigma}}_{1}
      -
      \mm{G}_{1} \VV{u}
      \right]
    ,\\
    \left(\mm{H} \otimes \mm{H}\right) \VV{{b}}_{2} &=
    \mm{G}_{2}^{T}
    \left[
      \mm{L}_{2} \VV{u}
      -
      \vv{\lambda}_{1}
      \right]
    +
    \mm{L}_{2}^{T}
    \left[
      \mm{H}\vv{\hat{\sigma}}_{2}
      -
      \mm{G}_{2} \VV{u}
      \right],\\
    \left(\mm{H} \otimes \mm{H}\right) \VV{{b}}_{3} &=
    \mm{G}_{3}^{T}
    \left[
      \mm{L}_{3} \VV{u}
      -
      \vv{\lambda}_{3}
      \right]
    +
    \mm{L}_{3}^{T}
    \left[
      \mm{H}\vv{\hat{\sigma}}_{3}
      -
      \mm{G}_{3} \VV{u}
      \right],\\
    \left(\mm{H} \otimes \mm{H}\right) \VV{{b}}_{4} &=
    \mm{G}_{4}^{T}
    \left[
      \mm{L}_{4} \VV{u}
      -
      \vv{\lambda}_{4}
      \right]
    +
    \mm{L}_{4}^{T}
    \left[
      \mm{H}\vv{\hat{\sigma}}_{4}
      -
      \mm{G}_{4} \VV{u}
      \right],
  \end{split}
\end{equation*}
where $\vv{\lambda}_{1}$, $\vv{\lambda}_{2}$, $\vv{\lambda}_{3}$, and
$\vv{\lambda}_{4}$ are the grid values of $\lambda$ along each of the four
faces. The yet-to-be-defined vectors $\mm{H}\vv{\hat{\sigma}}_{1}$,
$\mm{H}\vv{\hat{\sigma}}_{2}$, $\mm{H}\vv{\hat{\sigma}}_{3}$, and
$\mm{H}\vv{\hat{\sigma}}_{4}$ are (within the HDG literature) known as the
numerical fluxes and will be linear functions of the solution vector $\VV{u}$
and trace variables $\mm{\lambda}_{1}$, $\mm{\lambda}_{2}$, $\mm{\lambda}_{3}$,
and $\mm{\lambda}_{4}$.  We have scaled $\vv{\hat{\sigma}}_{k}$ by the matrix
$\mm{H}$ to highlight that these would be integrated flux terms in the HDG
literature and $\hat{\sigma}_{k}$ can be thought of as an approximation of
$\vv{\hat{n}}_{k} \cdot \mm{c} \hat{\nabla}u$.

Motivated by the hybridized symmetric interior penalty (IP-H) method
\citep{CockburnGopalakrishnanLazarov2009}, we take the penalty fluxes to be of
the form
\begin{align}
  \label{eqn:IP:flux}
  \mm{H}\vv{\hat{\sigma}}_{k} =
  \mm{G}_{k} \VV{u} - \mm{H} \mm{\tau}_{k}
  \left(\mm{L}_{k}\VV{u} - \vv{\lambda}_{k}\right);
\end{align}
thus $\mm{H}\vv{\hat{\sigma}}_{k}$ includes the norm-weighted boundary
derivative $\mm{G}_{k}$~\eref{eqn:weight:bnd:der} and
penalties related to the trace function $\lambda_k$.
Here $\mm{\tau}_{k}$ is a positive, diagonal matrix of penalty parameters, which
as we will see below, is required to be sufficiently large for the
local problem to be positive definite.

Multiplying \eref{eqn:disc} by $\mm{H} \otimes \mm{H}$, using the structure of
the derivative matrices \eref{eqn:A:structure}, and collecting all terms
involving $\VV{u}$ on the left-hand side gives a system of the form
\begin{subequations}\label{eqn:full:disc}
\begin{align}\label{eqn:disc:system}
  \left(\MM{A} + \MM{C}_{1} + \MM{C}_{2} + \MM{C}_{3} + \MM{C}_{4}\right) \VV{u} =
  \MM{M} \VV{u} = \VV{q}.
\end{align}
Here the left-hand side matrices are
\begin{align}
  \label{eqn:MMA}
  \MM{A} &= \MM{A}^{\left(c_{rr}\right)}_{rr}
    +\MM{A}^{\left(c_{ss}\right)}_{ss}
    +\MM{A}^{\left(c_{rs}\right)}_{rs}
    +\MM{A}^{\left(c_{sr}\right)}_{sr},\\
  \label{eqn:MMC}
  \MM{C}_{k} &= -\mm{L}_{k}^{T} \mm{G}_{k} - \mm{G}_{k}^{T} \mm{L}_{k} +
  \mm{L}_{k}^{T} \mm{H}\mm{\tau}_{k} \mm{L}_{k},
  \mbox{ for } k = 1,2,3,4,
\end{align}
and the right-hand side vector is
\begin{align}
  \label{eqn:qtilde}
  \VV{q} &= \left(\mm{H} \otimes \mm{H}\right) \MM{J} \VV{f}
  - \sum_{k=1}^{4} \mm{F}_{k}\vv{\lambda}_{k},
\end{align}
with the face matrix $\mm{F}_{k}$ being defined as
\begin{align}
  \label{eqn:mmFk}
  \mm{F}_{k} = \mm{G}_{k}^{T} - \mm{L}_{k}^{T}\mm{H} \mm{\tau}_{k};
\end{align}
the utility of defining $\mm{F}_{k}$ is a later connection with the structure of
the monolithic linear system \eref{eqn:full:system}.
\end{subequations}

The following theorem characterizes the structure of $\MM{M}$.
\begin{theorem}\label{thm:loc:PD}
  The local problem matrix $\MM{M}$ is symmetric positive definite if the
  components of the diagonal penalty matrices $\mm{\tau}_{k}$ for $k=1,2,3,4$
  are sufficiently large.
\end{theorem}
\begin{proof}
  See \aref{sec:app:loc:PD}
\end{proof}

\begin{remark}
  Explicit bounds for the penalty terms are given in the proof of
  Theorem~\ref{thm:loc:PD} given in \aref{sec:app:loc:PD};
  see~\eref{eqn:tau}. Since they are fairly complicated to
  state, we have chosen to omit them from the statement of the theorem.
\end{remark}

\begin{corollary}
  The local solution $\VV{u}$ is uniquely determined by $\MM{f}$,
  $\mm{\lambda}_{1}$, $\mm{\lambda}_{2}$, $\mm{\lambda}_{3}$, and
  $\mm{\lambda}_{4}$.
\end{corollary}
\begin{proof}
  Follows directly from Theorem~\ref{thm:loc:PD} since $\MM{f}$,
  $\mm{\lambda}_{1}$, $\mm{\lambda}_{2}$, $\mm{\lambda}_{3}$, and
  $\mm{\lambda}_{4}$ determine the right-hand side vector $\MM{q}$.
\end{proof}
\subsection{Global Problem}

We now turn to the global problem, namely the system that determines the trace
vector $\mm{\bar{\lambda}}$. To do this we let $\FF$ be the set of all block
faces with $\FF_{D}$ and $\FF_{N}$ being those faces that occur on the Dirichlet
and Neumann boundaries, respectively, and $\FF_{I}$ being the interior faces;
internal faces that both have a jump and those that do not are included in
$\FF_{I}$ with the latter having $\delta := 0$.  For each face $f \in \FF_{D}
\cup \FF_{N}$ we let the corresponding block and block face be $B_{f} \in \BB$
and $k_{f}$, respectively.  For each face $f \in \FF_{I}$ we let $B^{\pm}_{f}
\in \BB$ be the blocks connected to the two sides of the interface and let
$k^{\pm}_{f}$ be the connected sides of the blocks; for the jump interfaces the
plus- and minus-sides should correspond to those in \eref{eqn:bc:jump}. In what
follows the subscript $f$ is dropped when only one face $f \in \FF$ is being
considered.  Finally, for each $B \in \BB$ we let $\vv{\lambda}_{k} = \mm{P}_{B,k}
\vv{\bar{\lambda}}$, where $\mm{P}_{B,k}$ selects the values out of the global
vector of trace variables $\vv{\bar{\lambda}}$ that correspond to face $k$ and
block $B$.

\paragraph{Dirichlet Boundary Conditions:}
Consider face $f \in \FF_{D}$ which corresponds to face $k$ of block $B
\in \BB$. In this case we set $\vv{\lambda}_{k}$ in \eref{eqn:IP:flux} to be
\begin{align}
  \mm{\lambda}_{k} = \vv{g}_{D,f},
\end{align}
where $\vv{g}_{D,f}$ denotes the projection of $g_{D}$ to face $f$. With this
the penalty term $\MM{b}_{k}$ becomes
\begin{align}
  \label{eqn:numbc:Dirichlet}
  \left(\mm{H} \otimes \mm{H} \right) \MM{b}_{k}
  =
  \mm{F}_{k}
  \left(
  \mm{L}_{k} \VV{u}
  -
  \vv{g}_{D,k}
  \right),
\end{align}
which is penalization of the grid function along interface $k$ to the Dirichlet
boundary data. Since $\mm{\lambda}_{k}$ is determined independently of $\VV{u}$
and the structure of the matrix $\MM{M}$ remains unchanged.

\paragraph{Neumann Boundary Condition:}
Consider face $f \in \FF_{N}$ which corresponds to face $k$ of block $B
\in \BB$. In this case we require that $\mm{\lambda}_{k}$ in \eref{eqn:IP:flux}
satisfies
\begin{align*}
  \mm{H}\vv{\hat{\sigma}}_{k} = \mm{H} \mm{S}_{J,k} \vv{g}_{N,f},
\end{align*}
where $\vv{g}_{N,f}$ denotes the projection of $g_{N}$ to face $f$ and
$\mm{S}_{J,k}$ is a diagonal matrix of surface Jacobians along block face $k$.
As with the Dirichlet boundary condition, the variable $\mm{\lambda}_{k}$ can be
found uniquely in terms of the boundary data:
\begin{align}
  \label{eqn:numbc:Neumann}
  \vv{\lambda}_{k}
  =
  \mm{L}_{k}\VV{u}
  +
  \mm{\tau}_{k}^{-1}
  \left(
  \mm{S}_{J,k}
  \vv{g}_{N,k}
  -
  \mm{H}^{-1}
  \mm{G}_{k} \VV{u}
  \right),
\end{align}
which represents penalization of the boundary derivative towards the Neumann
boundary data. If $\mm{\lambda}_{k}$ is eliminated in this fashion from the
scheme, then $\MM{M}$ is modified as
\begin{align}
  \label{eqn:M:Neumann}
  \MM{M} := \MM{M} -
  \mm{F}_{k}
  \mm{H}^{-1}\mm{\tau}_{k}^{-1}
  \mm{F}_{k}^{T}.
\end{align}

\begin{theorem}\label{thm:Neumann:loc:PD}
  The modified local problem matrix $\MM{M}$ in \eref{eqn:M:Neumann} is
  symmetric positive definite if the components of the diagonal penalty matrices
  $\mm{\tau}_{k}$ for $k=1,2,3,4$ are sufficiently large and at least one face
  of the local block $B \in \BB$ is a Dirichlet boundary or interior interface.
\end{theorem}
\begin{proof}
  See \aref{sec:app:Neumann:loc:PD}
\end{proof}

\paragraph{Interfaces:}
We now consider an $f \in \FF_{I}$ which is connected to face $k^{\pm}$ of
blocks $B^{\pm}\in\BB$; below a subscript $B^{\pm}$ is added to denoted terms
associated with each block and a subscript $f, B^{\pm}$ for terms associated
with the respective faces of the blocks.  Continuity of the
solution and the $\mm{b}$-weighted normal derivative are enforced by
requiring
\begin{align}
  \label{eqn:locked:IC}
  \mm{H}\vv{\hat{\sigma}}_{f,B^{+}} +
  \mm{H}\vv{\hat{\sigma}}_{f,B^{-}} = \vv{0};
\end{align}
since $\vv{\hat{\sigma}}_{f,B^{\pm}}$ includes the outward pointing normal to
the blocks, condition \eref{eqn:locked:IC} implies that the terms are equal in
magnitude but opposite in sign. Using penalty formulation \eref{eqn:IP:flux} in
\eref{eqn:locked:IC} with
$\vv{\lambda}_{k}$ replaced with $\vv{\lambda}_{f} \mp \vv{\delta}_{f}/2$ gives
\begin{equation*}
  \begin{split}
    \vv{0} =\;& \left(\mm{G}_{f,B^{+}} \VV{u}_{B^{+}} + \mm{G}_{f,B^{-}} \VV{u}_{B^{-}}\right)\\
    &- \mm{H}\mm{\tau}_{f,B^{+}}\left(\mm{L}_{f,B^{+}} \VV{u}_{B^{+}} - \left(\mm{\lambda}_{f} - \frac{1}{2}\vv{\delta}_{f}\right) \right)\\
    &- \mm{H}\mm{\tau}_{f,B^{-}}\left(\mm{L}_{f,B^{-}} \VV{u}_{B^{-}} - \left(\mm{\lambda}_{f} + \frac{1}{2}\vv{\delta}_{f}\right) \right),
  \end{split}
\end{equation*}
where the first term represents penalization of the face normal derivative on
the two sides to the common value and the second two terms the penalization of
the $u_{B^{\pm}}$ to $\lambda_f \mp \delta_f/2$. By grouping terms, the
above equation can be rewritten as
\begin{align}
  \label{eqn:global:face:couple}
  \mm{F}_{f,B^{+}}^{T} \VV{u}_{B^{+}}
  +
  \mm{F}_{f,B^{-}}^{T} \VV{u}_{B^{-}}
  +
  \mm{D}_{f} \mm{\lambda}_{f}
  =
  \frac{1}{2}\mm{H}\left(
  \mm{\tau}_{f, B^{+}}
  -
  \mm{\tau}_{f, B^{-}}
  \right)
  \vv{\delta}_{f}.
\end{align}
Here the matrices $\mm{F}_{f,B^{\pm}}$ are defined by \eref{eqn:mmFk} and
the diagonal matrix $\mm{D}_{f}$ is
\begin{align*}
  \mm{D}_{f} = \mm{H}\left(
  \mm{\tau}_{f, B^{+}}
  +
  \mm{\tau}_{f, B^{-}}
  \right).
\end{align*}

With this, all the terms in linear system \eref{eqn:full:system} can be defined.
The solution vector and trace vectors are
\begin{align*}
  \vv{\bar{u}} &=
  \begin{bmatrix}
    \VV{u}_{1}\\
    \VV{u}_{2}\\
    \vdots\\
    \VV{u}_{N_{b}}
  \end{bmatrix},&
  \vv{\bar{\lambda}} &=
  \begin{bmatrix}
    \vv{\lambda}_{1}\\
    \vv{\lambda}_{2}\\
    \vdots\\
    \vv{\lambda}_{N_{I}}
  \end{bmatrix},
\end{align*}
with $N_{I}$ being the number of interfaces.
Multiplying out the terms in \eref{eqn:full:system} gives
\begin{equation*}
  \begin{split}
    \mm{\bar{M}} \vv{\bar{u}} +  \mm{\bar{F}} \vv{\bar{\lambda}} &= \vv{\bar{g}},\\
    \mm{\bar{F}}^{T} \vv{\bar{u}} + \mm{\bar{D}}\vv{\bar{\lambda}} &=
    \vv{\bar{g}}_{\delta}.
  \end{split}
\end{equation*}
This form, along with the definition of the local problem
\eref{eqn:full:disc} and the coupling equation
\eref{eqn:global:face:couple}, implies that the matrices $\mm{\bar{M}}$ and
$\mm{\bar{D}}$ are
\begin{align*}
  \mm{\bar{M}} &=
  \begin{bmatrix}
    \MM{M}_{1}\\
    & \MM{M}_{2}\\
    && \ddots\\
    &&& \MM{M}_{N_{b}}
  \end{bmatrix},&
  \mm{\bar{D}} &=
  \begin{bmatrix}
    \MM{D}_{1}\\
    & \MM{D}_{2}\\
    && \ddots\\
    &&& \MM{D}_{N_{I}}
  \end{bmatrix}.
\end{align*}
Furthermore, since each matrix $\MM{D}_{f}$ is diagonal, the matrix
$\mm{\bar{D}}$ is also diagonal. To write down the form of $\mm{\bar{F}}$ it is
convenient to think of it as a block matrix with sub-matrix $fB$ being the
columns associated with interface $f$ and rows associated with block $B$. Thus,
block $\mm{\bar{F}}_{fB}$ is zero unless block $B$ is connected to interface $f$
through local face $k_{f}$ in which case
\begin{align*}
  \mm{\bar{F}}_{fB} = \mm{F}_{k_{f}, B}.
\end{align*}
The right-hand side vector $\vv{\bar{g}}$ is defined from the boundary data
using \eref{eqn:numbc:Dirichlet} and \eref{eqn:numbc:Neumann}, and similarly
$\vv{\bar{g}}_{\delta}$ is defined from the right-hand side of
\eref{eqn:global:face:couple}.

In order to prove the positive definiteness of the coupled system, we first note
that $\mm{\bar{M}}$ and $\mm{\bar{D}}$ are symmetric positive definite since
they are block diagonal matrices formed from symmetric positive definite
matrices. If the trace variables $\vv{\bar{\lambda}}$ are eliminated using the
Schur complement of the $\mm{\bar{D}}$ block the system for $\vv{\bar{u}}$, the
resulting system is
\begin{align}\label{eqn:schur:system:D}
  \left(\mm{\bar{M}} - \mm{\bar{F}} \mm{\bar{D}}^{-1} \mm{\bar{F}}^{T}\right)
  \vv{\bar{u}} = \vv{\bar{g}} - \mm{\bar{F}} \mm{\bar{D}}^{-1}
  \vv{\bar{g}}_{\delta}.
\end{align}
This corresponds to the elimination of the trace variables by solving the
coupling relation \eref{eqn:global:face:couple} for $\vv{\lambda}_{f}$ and
substituting in the local problem \eref{eqn:full:disc} for each block.
The matrix on the left-hand side of \eref{eqn:schur:system:D} is characterized
by the following theorem which says that if the individual local problems are
symmetric positive definite, then the coupled problem is symmetric positive
definite.
\begin{theorem}\label{thm:coupled:PD}
  The matrix $\mm{\bar{M}} - \mm{\bar{F}} \mm{\bar{D}}^{-1} \mm{\bar{F}}^{T}$ is
  symmetric positive definite as long as the penalty matrices $\vv{\tau}_{k,B}$
  for $k = 1,2,3,4$ and $B \in \BB$ are sufficiently large that each
  $\MM{M}_{B}$ is positive definite.
\end{theorem}
\begin{proof}
  See \aref{sec:app:coupled:PD}
\end{proof}

The following corollary characterizes the global system and the Schur complement
of the $\mm{\bar{M}}$ block of the global system.
\begin{corollary}\label{cor:coupled:PD}
  The global system matrix
  \begin{align}
    \begin{bmatrix}
      \mm{\bar{M}} & \mm{\bar{F}}\\
      \mm{\bar{F}}^{T} & \mm{\bar{D}}
    \end{bmatrix}
  \end{align}
  and the Schur complement of $\mm{\bar{M}}$ block,
  $\mm{\bar{D}} - \mm{\bar{F}}^{T} \mm{\bar{M}}^{-1} \mm{\bar{F}}$, are
  symmetric positive definite.
\end{corollary}
\begin{proof}
  See \aref{sec:app:coupled:PD}
\end{proof}
\section{Numerical Results}\label{sec:numerical}
We now confirm the theoretical results concerning the positive definiteness of
the system, the bounds on the penalty parameters, and numerically investigate
accuracy of the hybridized technique. All of the solves in this section are
done using direct solves. The Julia
\citep{bezanson2017julia}\footnote{Simulations run with Julia 1.5.3} codes used
to generate the numerical results are available at
\url{https://doi.org/10.5281/zenodo.4716826}.

\subsection{Positive Definiteness of the Local and Global Problems}
We begin by confirming that the local problem with both Dirichlet and Neumann
boundary conditions is symmetric positive definite. To do this, we consider a
single block, and assign a pseudo-random generated symmetric positive definite
coefficient matrix $\mm{c}$ at each grid point. The blocks are taken to use
grids of size $N \times N = (3p+2) \times (3p+2)$ where $2p$ is the interior
order of the SBP operator.

\begin{figure}
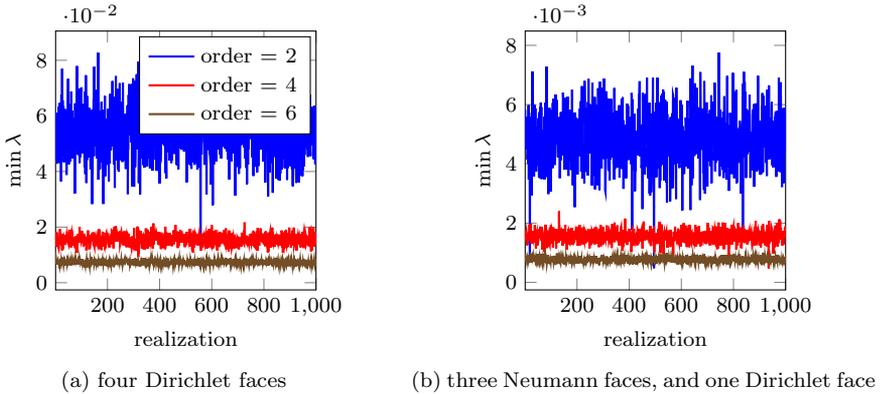

  \begin{subfigure}[t]{0.5\textwidth}
    \centering

     \caption{four Dirichlet faces}
  \end{subfigure}
  \begin{subfigure}[t]{0.5\textwidth}
    \centering
%
     \caption{three Neumann faces, and one Dirichlet face}
  \end{subfigure}
  \caption{Plot of the minimum eigenvalue of the local operator for $1000$
  psuedo-randomly assigned sets of coefficient matrix values for SBP operators
  with interior orders $2$ (blue line), $4$ (red line), and $6$ (brown
  line).\label{fig:min:eig:local}}
\end{figure}
To confirm that the operator is positive definite we compute $1000$
realizations of the pseudo-random coefficients and numerically compute the
minimum eigenvalue with the penalty parameter defined by the equality version of
\eref{eqn:tau}. Two sets of boundary conditions are considered: (1) when all
four faces of the block are Dirichlet and (2) when three faces are Neumann and
one face is Dirichlet.  The result of these calculations are shown in
Figure~\ref{fig:min:eig:local}. From this we see that the system is positive
definite. One thing of note is that the local system with Neumann boundary
conditions has a minimum eigenvalue which is an order of magnitude lower than
the purely Dirichlet case. Though not shown, when all four boundaries are
Neumann the minimum computed eigenvalue is $\sim 10^{-16}$--$10^{-14}$. This
conforms with the theory since in this case the system should be singular. An
important implication of Figure~\ref{fig:min:eig:local} is that the bound on the
penalty parameter given in \eref{eqn:tau} is not tight for all cases.

\begin{figure}
  \begin{subfigure}[t]{0.5\textwidth}
    \centering
\begin{tikzpicture}[]
\begin{semilogxaxis}[legend pos = {south east}, ylabel = {$\min \lambda$}, xmin = {0.01}, xmax = {10.0}, xlabel = {$\tau_{s}$}, width=5cm, height=5cm]\addplot+ [no marks, thick]coordinates {
(0.01, -1.9052385301206902)
(0.010974987654930561, -1.8870356782692823)
(0.012045035402587823, -1.867622669344128)
(0.013219411484660288, -1.8469752267581672)
(0.014508287784959394, -1.8250760070718983)
(0.015922827933410922, -1.8019138645984876)
(0.01747528400007684, -1.777482227325204)
(0.019179102616724886, -1.7517765428043035)
(0.02104904144512021, -1.7247908971011219)
(0.02310129700083159, -1.6965140646473507)
(0.025353644939701114, -1.6669253604968974)
(0.027825594022071243, -1.635990688594844)
(0.030538555088334154, -1.6036590890980746)
(0.033516026509388425, -1.5698599121733687)
(0.03678379771828634, -1.534500555335412)
(0.040370172585965536, -1.4974645761953809)
(0.044306214575838804, -1.4586099810202506)
(0.04862601580065353, -1.4177675855414997)
(0.0533669923120631, -1.3747394946325875)
(0.05857020818056667, -1.3292978831137483)
(0.06428073117284322, -1.2811843332502797)
(0.07054802310718641, -1.230109989989349)
(0.0774263682681127, -1.175756768101321)
(0.08497534359086442, -1.1177798451813172)
(0.093260334688322, -1.05581176645598)
(0.10235310218990261, -0.9894687467400681)
(0.11233240329780274, -0.9183603221349256)
(0.12328467394420663, -0.8421048402728246)
(0.1353047774579807, -0.7603580463044448)
(0.1484968262254465, -0.6729390313311339)
(0.16297508346206444, -0.6007049619805105)
(0.17886495290574345, -0.53196823572371)
(0.19630406500402708, -0.4574604362588097)
(0.2154434690031884, -0.3767012487439427)
(0.23644894126454072, -0.2950426275758613)
(0.25950242113997357, -0.21221495045268676)
(0.2848035868435802, -0.12143271124379264)
(0.3125715849688236, -0.021932364038521328)
(0.34304692863149183, 0.08711685058906514)
(0.37649358067924676, 0.1248328190114053)
(0.4132012400115337, 0.12535940061706222)
(0.4534878508128582, 0.12568629516008523)
(0.4977023564332111, 0.12592068183164962)
(0.5462277217684341, 0.12610292028540862)
(0.599484250318941, 0.12625242072606016)
(0.657933224657568, 0.12637989325490445)
(0.7220809018385465, 0.1264918082113814)
(0.7924828983539174, 0.12659233595262717)
(0.8697490026177833, 0.12668429653853241)
(0.9545484566618341, 0.1267696664963141)
(1.0476157527896648, 0.12684986779678784)
(1.1497569953977358, 0.12692594172173183)
(1.2618568830660204, 0.12699865833095583)
(1.384886371393873, 0.12706858821866626)
(1.5199110829529336, 0.12713615135632228)
(1.6681005372000586, 0.1272016515854411)
(1.8307382802953682, 0.12726530189084523)
(2.0092330025650473, 0.12732724361257255)
(2.2051307399030455, 0.12738756157723227)
(2.420128264794382, 0.12744629640455163)
(2.6560877829466865, 0.12750345478430036)
(2.9150530628251765, 0.12755901822447369)
(3.1992671377973836, 0.1276129505788102)
(3.511191734215131, 0.12766520453840435)
(3.8535285937105295, 0.12771572719726967)
(4.229242874389499, 0.12776446475724365)
(4.641588833612778, 0.12781136641603758)
(5.094138014816378, 0.12785638747690314)
(5.590810182512223, 0.1278994917213422)
(6.1359072734131725, 0.12794065309390085)
(6.7341506577508214, 0.1279798567597949)
(7.390722033525779, 0.12801709960386098)
(8.11130830789687, 0.12805239024784673)
(8.902150854450387, 0.12808574866714414)
(9.770099572992255, 0.12811720548792668)
(10.722672220103231, 0.128146801044239)
(11.768119524349984, 0.12817458426878872)
(12.91549665014884, 0.12820061148375633)
(14.174741629268055, 0.12822494514917812)
(15.556761439304715, 0.12824765261638926)
(17.073526474706906, 0.12826880492600237)
(18.73817422860384, 0.12828847567809235)
(20.565123083486515, 0.12830673999524825)
(22.5701971963392, 0.12832367359256278)
(24.770763559917114, 0.12833935196050947)
(27.1858824273294, 0.12835384966164404)
(29.836472402833387, 0.12836723973895536)
(32.74549162877728, 0.12837959323136636)
(35.938136638046274, 0.1283909787877024)
(39.44206059437656, 0.12840146237015135)
(43.287612810830595, 0.12841110703745137)
(47.50810162102796, 0.12841997279930004)
(52.14008287999685, 0.12842811653026595)
(57.22367659350217, 0.12843559193634477)
(62.80291441834253, 0.12844244956333106)
(68.92612104349699, 0.12844873684012487)
(75.64633275546291, 0.1284544981527838)
(83.02175681319744, 0.12845977493618388)
(91.11627561154891, 0.1284646057857808)
(100.0, 0.12846902658090054)
};
\addlegendentry{order = 2}
\addplot+ [no marks, thick]coordinates {
(0.01, -2.8623812966911175)
(0.010974987654930561, -2.8189163001340374)
(0.012045035402587823, -2.7714231702022087)
(0.013219411484660288, -2.7195399999667824)
(0.014508287784959394, -2.6628742995028065)
(0.015922827933410922, -2.6010009817299604)
(0.01747528400007684, -2.533460436501615)
(0.019179102616724886, -2.4597567923106025)
(0.02104904144512021, -2.384713898559432)
(0.02310129700083159, -2.319881146031947)
(0.025353644939701114, -2.2501396932052673)
(0.027825594022071243, -2.17498308531174)
(0.030538555088334154, -2.093886432769337)
(0.033516026509388425, -2.00630929118263)
(0.03678379771828634, -1.9117004538662157)
(0.040370172585965536, -1.809505555677217)
(0.044306214575838804, -1.6991786634141135)
(0.04862601580065353, -1.5801994286424517)
(0.0533669923120631, -1.4520979417604571)
(0.05857020818056667, -1.3144902005686545)
(0.06428073117284322, -1.167128158352864)
(0.07054802310718641, -1.0099700095382314)
(0.0774263682681127, -0.9364509745877636)
(0.08497534359086442, -0.8782991513798312)
(0.093260334688322, -0.8148728322205027)
(0.10235310218990261, -0.7457607436267625)
(0.11233240329780274, -0.6705489591875508)
(0.12328467394420663, -0.5888355643863238)
(0.1353047774579807, -0.5002541178581275)
(0.1484968262254465, -0.4045108585316589)
(0.16297508346206444, -0.3014438125876804)
(0.17886495290574345, -0.1911178694493456)
(0.19630406500402708, -0.0739839189866485)
(0.2154434690031884, 0.04279553684539433)
(0.23644894126454072, 0.042986744954772504)
(0.25950242113997357, 0.04303748090978365)
(0.2848035868435802, 0.043078365661606395)
(0.3125715849688236, 0.043113076478685265)
(0.34304692863149183, 0.04314318307745829)
(0.37649358067924676, 0.04316962956785509)
(0.4132012400115337, 0.043193064745851856)
(0.4534878508128582, 0.043213966524198184)
(0.4977023564332111, 0.04323270245630448)
(0.5462277217684341, 0.04324956384691332)
(0.599484250318941, 0.04326478685427198)
(0.657933224657568, 0.043278566399443524)
(0.7220809018385465, 0.04329106576448649)
(0.7924828983539174, 0.043302423449697806)
(0.8697490026177833, 0.04331275820815333)
(0.9545484566618341, 0.04332217282402592)
(1.0476157527896648, 0.043330757000092146)
(1.1497569953977358, 0.043338589598067856)
(1.2618568830660204, 0.0433457403994917)
(1.384886371393873, 0.04335227150526029)
(1.5199110829529336, 0.04335823845947791)
(1.6681005372000586, 0.04336369116001905)
(1.8307382802953682, 0.04336867460295742)
(2.0092330025650473, 0.04337322949620381)
(2.2051307399030455, 0.043377392769809814)
(2.420128264794382, 0.043381198003842845)
(2.6560877829466865, 0.043384675790272574)
(2.9150530628251765, 0.04338785404187028)
(3.1992671377973836, 0.04339075825797573)
(3.511191734215131, 0.04339341175532603)
(3.8535285937105295, 0.04339583587002403)
(4.229242874389499, 0.043398050135877726)
(4.641588833612778, 0.04340007244279371)
(5.094138014816378, 0.04340191917844549)
(5.590810182512223, 0.04340360535595209)
(6.1359072734131725, 0.043405144729040716)
(6.7341506577508214, 0.04340654989688004)
(7.390722033525779, 0.04340783239959612)
(8.11130830789687, 0.04340900280535196)
(8.902150854450387, 0.04341007079036578)
(9.770099572992255, 0.043411045212313114)
(10.722672220103231, 0.04341193417756152)
(11.768119524349984, 0.04341274510302647)
(12.91549665014884, 0.043413484773391284)
(14.174741629268055, 0.043414159393206815)
(15.556761439304715, 0.04341477463541776)
(17.073526474706906, 0.043415335685498775)
(18.73817422860384, 0.04341584728239661)
(20.565123083486515, 0.04341631375597186)
(22.5701971963392, 0.04341673906164748)
(24.770763559917114, 0.043417126811950785)
(27.1858824273294, 0.04341748030570722)
(29.836472402833387, 0.043417802554780026)
(32.74549162877728, 0.04341809630865288)
(35.938136638046274, 0.04341836407669732)
(39.44206059437656, 0.04341860814908654)
(43.287612810830595, 0.043418830615562405)
(47.50810162102796, 0.0434190333826677)
(52.14008287999685, 0.043419218189825665)
(57.22367659350217, 0.04341938662352842)
(62.80291441834253, 0.043419540130997775)
(68.92612104349699, 0.04341968003203127)
(75.64633275546291, 0.04341980753019589)
(83.02175681319744, 0.043419923723026284)
(91.11627561154891, 0.043420029611313055)
(100.0, 0.043420126107546454)
};
\addlegendentry{order = 4}
\addplot+ [no marks, thick]coordinates {
(0.01, -4.029778203128789)
(0.010974987654930561, -3.991728082789739)
(0.012045035402587823, -3.9503857119640227)
(0.013219411484660288, -3.9054303907538355)
(0.014508287784959394, -3.856513305108747)
(0.015922827933410922, -3.8032558710569893)
(0.01747528400007684, -3.7452481049316138)
(0.019179102616724886, -3.682047079970385)
(0.02104904144512021, -3.613175554316311)
(0.02310129700083159, -3.538120888079984)
(0.025353644939701114, -3.4563344107367833)
(0.027825594022071243, -3.367231458915912)
(0.030538555088334154, -3.270192384540014)
(0.033516026509388425, -3.164564942800627)
(0.03678379771828634, -3.0496686209713646)
(0.040370172585965536, -2.9248016808144754)
(0.044306214575838804, -2.7892519868376935)
(0.04862601580065353, -2.642313122415657)
(0.0533669923120631, -2.4833079235300666)
(0.05857020818056667, -2.311622497634998)
(0.06428073117284322, -2.126755236501278)
(0.07054802310718641, -1.928387627914887)
(0.0774263682681127, -1.716487489027934)
(0.08497534359086442, -1.4914618948490554)
(0.093260334688322, -1.2543891116308474)
(0.10235310218990261, -1.0073805759021077)
(0.11233240329780274, -0.7541576292161327)
(0.12328467394420663, -0.5009381846525852)
(0.1353047774579807, -0.25746638349990236)
(0.1484968262254465, -0.03680730382444479)
(0.16297508346206444, 0.022732859367688128)
(0.17886495290574345, 0.022754953797570924)
(0.19630406500402708, 0.02276936578761845)
(0.2154434690031884, 0.022780356365888266)
(0.23644894126454072, 0.022789230974410567)
(0.25950242113997357, 0.022796626425638553)
(0.2848035868435802, 0.02280291503848919)
(0.3125715849688236, 0.022808338217658923)
(0.34304692863149183, 0.022813063612308344)
(0.37649358067924676, 0.022817213539973534)
(0.4132012400115337, 0.022820880629868773)
(0.4534878508128582, 0.022824137096718503)
(0.4977023564332111, 0.022827040562923225)
(0.5462277217684341, 0.022829637885816656)
(0.599484250318941, 0.0228319677698702)
(0.657933224657568, 0.022834062607271555)
(0.7220809018385465, 0.022835949809723716)
(0.7924828983539174, 0.022837652795321087)
(0.8697490026177833, 0.022839191734188503)
(0.9545484566618341, 0.022840584122289877)
(1.0476157527896648, 0.02284184522983421)
(1.1497569953977358, 0.02284298845657018)
(1.2618568830660204, 0.02284402561684671)
(1.384886371393873, 0.02284496717079011)
(1.5199110829529336, 0.022845822413819505)
(1.6681005372000586, 0.02284659963323838)
(1.8307382802953682, 0.022847306238880535)
(2.0092330025650473, 0.02284794887281231)
(2.2051307399030455, 0.022848533502237638)
(2.420128264794382, 0.022849065498532266)
(2.6560877829466865, 0.022849549705108572)
(2.9150530628251765, 0.022849990495858844)
(3.1992671377973836, 0.02285039182580078)
(3.511191734215131, 0.022850757275322993)
(3.8535285937105295, 0.022851090088830628)
(4.229242874389499, 0.02285139320873462)
(4.641588833612778, 0.022851669305591942)
(5.094138014816378, 0.022851920804617912)
(5.590810182512223, 0.02285214990955273)
(6.1359072734131725, 0.022852358623659114)
(6.7341506577508214, 0.022852548768709735)
(7.390722033525779, 0.022852722002058967)
(8.11130830789687, 0.022852879831707223)
(8.902150854450387, 0.022853023630295184)
(9.770099572992255, 0.022853154647320257)
(10.722672220103231, 0.0228532740205106)
(11.768119524349984, 0.022853382785831707)
(12.91549665014884, 0.02285348188681088)
(14.174741629268055, 0.022853572182777677)
(15.556761439304715, 0.02285365445651862)
(17.073526474706906, 0.02285372942106512)
(18.73817422860384, 0.022853797726001216)
(20.565123083486515, 0.02285385996307528)
(22.5701971963392, 0.022853916671463078)
(24.770763559917114, 0.022853968342334934)
(27.1858824273294, 0.022854015423251682)
(29.836472402833387, 0.022854058321934816)
(32.74549162877728, 0.022854097409954286)
(35.938136638046274, 0.0228541330257827)
(39.44206059437656, 0.02285416547787291)
(43.287612810830595, 0.022854195047234932)
(47.50810162102796, 0.022854221989946834)
(52.14008287999685, 0.02285424653931831)
(57.22367659350217, 0.022854268907958133)
(62.80291441834253, 0.02285428928957141)
(68.92612104349699, 0.022854307860648758)
(75.64633275546291, 0.022854324782031255)
(83.02175681319744, 0.02285434020026698)
(91.11627561154891, 0.022854354248849393)
(100.0, 0.02285436704946132)
};
\addlegendentry{order = 6}
\end{semilogxaxis}

\end{tikzpicture}
     \caption{Minimum eigenvalues for increasing
    $\tau_s$\label{fig:tau:scaling:min}}
  \end{subfigure}
  \begin{subfigure}[t]{0.5\textwidth}
    \centering
\begin{tikzpicture}[]
\begin{loglogaxis}[legend pos = {north west}, ylabel = {$\max \lambda$}, xmin = {0.01}, xmax = {10.0}, xlabel = {$\tau_{s}$}, width=5cm, height=5cm]\addplot+ [no marks, thick]coordinates {
(0.01, 4.493913292360292)
(0.010974987654930561, 4.494235419336076)
(0.012045035402587823, 4.738023184124386)
(0.013219411484660288, 5.174568784003421)
(0.014508287784959394, 5.694897568544699)
(0.015922827933410922, 6.374991649823292)
(0.01747528400007684, 7.130879221273564)
(0.019179102616724886, 7.964660793254656)
(0.02104904144512021, 8.882804149346534)
(0.02310129700083159, 9.89288560817841)
(0.025353644939701114, 11.003419200533504)
(0.027825594022071243, 12.223863958371846)
(0.030538555088334154, 13.564676854847606)
(0.033516026509388425, 15.037387837508614)
(0.03678379771828634, 16.654690362519474)
(0.040370172585965536, 18.43054533060277)
(0.044306214575838804, 20.380297958172072)
(0.04862601580065353, 22.520807867517473)
(0.0533669923120631, 24.870593101397077)
(0.05857020818056667, 27.44998905070083)
(0.06428073117284322, 30.281323505736648)
(0.07054802310718641, 33.38910923715228)
(0.0774263682681127, 36.80025569940358)
(0.08497534359086442, 40.54430163808445)
(0.093260334688322, 44.653670578698055)
(0.10235310218990261, 49.16395138302186)
(0.11233240329780274, 54.1142062835598)
(0.12328467394420663, 59.54730904968572)
(0.1353047774579807, 65.51031620373152)
(0.1484968262254465, 72.05487449419337)
(0.16297508346206444, 79.23766814922767)
(0.17886495290574345, 87.12090977960139)
(0.19630406500402708, 95.77287917943846)
(0.2154434690031884, 105.26851468875817)
(0.23644894126454072, 115.6900622377224)
(0.25950242113997357, 127.1277876925582)
(0.2848035868435802, 139.6807586718135)
(0.3125715849688236, 153.45770260358006)
(0.34304692863149183, 168.57794845490733)
(0.37649358067924676, 185.17246028955293)
(0.4132012400115337, 203.38497160568383)
(0.4534878508128582, 223.3732302782)
(0.4977023564332111, 245.31036488842082)
(0.5462277217684341, 269.38638427533283)
(0.599484250318941, 295.80982329658235)
(0.657933224657568, 324.8095490538034)
(0.7220809018385465, 356.6367432268054)
(0.7924828983539174, 391.5670776865192)
(0.8697490026177833, 429.90310223071054)
(0.9545484566618341, 471.9768651238331)
(1.0476157527896648, 518.1527891387846)
(1.1497569953977358, 568.8308280114197)
(1.2618568830660204, 624.4499306474883)
(1.384886371393873, 685.4918430872397)
(1.5199110829529336, 752.4852811584909)
(1.6681005372000586, 826.0105099596154)
(1.8307382802953682, 906.7043698377852)
(2.0092330025650473, 995.2657923950039)
(2.2051307399030455, 1092.4618542989656)
(2.420128264794382, 1199.1344213339182)
(2.6560877829466865, 1316.207440239141)
(2.9150530628251765, 1444.6949414933952)
(3.1992671377973836, 1585.709822361692)
(3.511191734215131, 1740.47348627886)
(3.8535285937105295, 1910.3264220617061)
(4.229242874389499, 2096.739814581771)
(4.641588833612778, 2301.3282874648116)
(5.094138014816378, 2525.8638881881743)
(5.590810182512223, 2772.2914367082617)
(6.1359072734131725, 3042.745370560557)
(6.7341506577508214, 3339.56823233645)
(7.390722033525779, 3665.330959666461)
(8.11130830789687, 4022.8551534520693)
(8.902150854450387, 4415.237517222881)
(9.770099572992255, 4845.876679301142)
(10.722672220103231, 5318.5026300943255)
(11.768119524349984, 5837.2090294877535)
(12.91549665014884, 6406.488664168055)
(14.174741629268055, 7031.272361992131)
(15.556761439304715, 7716.971700459193)
(17.073526474706906, 8469.525879206229)
(18.73817422860384, 9295.453162513955)
(20.565123083486515, 10201.90733739375)
(22.5701971963392, 11196.739676268391)
(24.770763559917114, 12288.566940937886)
(27.1858824273294, 13486.84601684788)
(29.836472402833387, 14801.955824107481)
(32.74549162877728, 16245.287214729775)
(35.938136638046274, 17829.341634742916)
(39.44206059437656, 19567.839405734987)
(43.287612810830595, 21475.83856371573)
(47.50810162102796, 23569.86528462109)
(52.14008287999685, 25868.057026143095)
(57.22367659350217, 28390.31962571152)
(62.80291441834253, 31158.499715334878)
(68.92612104349699, 34196.573946675744)
(75.64633275546291, 37530.85666533741)
(83.02175681319744, 41190.227833136385)
(91.11627561154891, 45206.38317251633)
(100.0, 49614.10869973244)
};
\addlegendentry{order = 2}
\addplot+ [no marks, thick]coordinates {
(0.01, 25.79212500596308)
(0.010974987654930561, 28.422703846010897)
(0.012045035402587823, 31.31308241781908)
(0.013219411484660288, 34.488142074857414)
(0.014508287784959394, 37.975263473240716)
(0.015922827933410922, 41.80455590822327)
(0.01747528400007684, 46.00911223924591)
(0.019179102616724886, 50.62529109050448)
(0.02104904144512021, 55.69302839723864)
(0.02310129700083159, 61.25618072575127)
(0.025353644939701114, 67.36290314382762)
(0.027825594022071243, 74.06606476989127)
(0.030538555088334154, 81.42370549324579)
(0.033516026509388425, 89.49953774161642)
(0.03678379771828634, 98.36349758226832)
(0.040370172585965536, 108.09234988489062)
(0.044306214575838804, 118.7703527535223)
(0.04862601580065353, 130.48998695619946)
(0.0533669923120631, 143.352756650062)
(0.05857020818056667, 157.47006832172994)
(0.06428073117284322, 172.96419554372642)
(0.07054802310718641, 189.96933789366503)
(0.0774263682681127, 208.63278320055684)
(0.08497534359086442, 229.11618317915836)
(0.093260334688322, 251.59695349660615)
(0.10235310218990261, 276.2698103943297)
(0.11233240329780274, 303.34845717175403)
(0.12328467394420663, 333.0674351369396)
(0.1353047774579807, 365.6841550542864)
(0.1484968262254465, 401.4811266832186)
(0.16297508346206444, 440.76840571778683)
(0.17886495290574345, 483.8862793204217)
(0.19630406500402708, 531.2082135098291)
(0.2154434690031884, 583.144087931237)
(0.23644894126454072, 640.1437460264776)
(0.25950242113997357, 702.7008913533703)
(0.2848035868435802, 771.3573638020601)
(0.3125715849688236, 846.7078327465999)
(0.34304692863149183, 929.4049477813606)
(0.37649358067924676, 1020.164991655268)
(0.4132012400115337, 1119.7740843667475)
(0.4534878508128582, 1229.0949921561262)
(0.4977023564332111, 1349.0746003716308)
(0.5462277217684341, 1480.7521149352638)
(0.599484250318941, 1625.2680634457013)
(0.657933224657568, 1783.8741738813487)
(0.7220809018385465, 1957.9442164681134)
(0.7924828983539174, 2148.985902618986)
(0.8697490026177833, 2358.6539440082065)
(0.9545484566618341, 2588.764384891563)
(1.0476157527896648, 2841.3103318123826)
(1.1497569953977358, 3118.4792169365346)
(1.2618568830660204, 3422.6717445429763)
(1.384886371393873, 3756.5226847754766)
(1.5199110829529336, 4122.9236947608515)
(1.6681005372000586, 4525.0483647594865)
(1.8307382802953682, 4966.379706285756)
(2.0092330025650473, 5450.740320287367)
(2.2051307399030455, 5982.325506685973)
(2.420128264794382, 6565.739602057956)
(2.6560877829466865, 7206.0358601951775)
(2.9150530628251765, 7908.760220971686)
(3.1992671377973836, 8679.999346621884)
(3.511191734215131, 9526.433341496773)
(3.8535285937105295, 10455.393611932286)
(4.229242874389499, 11474.926367383483)
(4.641588833612778, 12593.862312841675)
(5.094138014816378, 13821.893136176499)
(5.590810182512223, 15169.65545289977)
(6.1359072734131725, 16648.822935440603)
(6.7341506577508214, 18272.207424911398)
(7.390722033525779, 20053.869901146427)
(8.11130830789687, 22009.242272182488)
(8.902150854450387, 24155.261038063945)
(9.770099572992255, 26510.51398670367)
(10.722672220103231, 29095.40119240977)
(11.768119524349984, 31932.31171156994)
(12.91549665014884, 35045.81750594732)
(14.174741629268055, 38462.88627325836)
(15.556761439304715, 42213.11502846891)
(17.073526474706906, 46328.986458978856)
(18.73817422860384, 50846.15027412169)
(20.565123083486515, 55803.73198589189)
(22.5701971963392, 61244.671795411625)
(24.770763559917114, 67216.09652041201)
(27.1858824273294, 73769.72778518243)
(29.836472402833387, 80962.33000853712)
(32.74549162877728, 88856.20207005512)
(35.938136638046274, 97519.71691317279)
(39.44206059437656, 107027.91375890946)
(43.287612810830595, 117463.14805970248)
(47.50810162102796, 128915.80482293874)
(52.14008287999685, 141485.0814826485)
(57.22367659350217, 155279.84710022312)
(62.80291441834253, 170419.58533614426)
(68.92612104349699, 187035.42936029023)
(75.64633275546291, 205271.29766472007)
(83.02175681319744, 225285.14061680305)
(91.11627561154891, 247250.30854973898)
(100.0, 271357.05324021506)
};
\addlegendentry{order = 4}
\addplot+ [no marks, thick]coordinates {
(0.01, 229.3328007496142)
(0.010974987654930561, 251.75517100088086)
(0.012045035402587823, 276.36394548995315)
(0.013219411484660288, 303.37227254423397)
(0.014508287784959394, 333.0140844641553)
(0.015922827933410922, 365.5461236971589)
(0.01747528400007684, 401.2501665889083)
(0.019179102616724886, 440.435463971526)
(0.02104904144512021, 483.4414197270271)
(0.02310129700083159, 530.6405305255756)
(0.025353644939701114, 582.4416122005103)
(0.027825594022071243, 639.2933407048846)
(0.030538555088334154, 701.6881383192458)
(0.033516026509388425, 770.166438770693)
(0.03678379771828634, 845.3213682053885)
(0.040370172585965536, 927.8038825586568)
(0.044306214575838804, 1018.3284058198974)
(0.04862601580065353, 1117.6790180280982)
(0.0533669923120631, 1226.716246595337)
(0.05857020818056667, 1346.3845197813264)
(0.06428073117284322, 1477.720346877312)
(0.07054802310718641, 1621.861295952191)
(0.0774263682681127, 1780.0558469216708)
(0.08497534359086442, 1953.6742052829645)
(0.093260334688322, 2144.2201701784616)
(0.10235310218990261, 2353.3441595837735)
(0.11233240329780274, 2582.857505438122)
(0.12328467394420663, 2834.7481425345486)
(0.1353047774579807, 3111.197827059738)
(0.1484968262254465, 3414.6010339219233)
(0.16297508346206444, 3747.585696546883)
(0.17886495290574345, 4113.035968779817)
(0.19630406500402708, 4514.117206045994)
(0.2154434690031884, 4954.30338214502)
(0.23644894126454072, 5437.407179149842)
(0.25950242113997357, 5967.613011034829)
(0.2848035868435802, 6549.513267067663)
(0.3125715849688236, 7188.14808888822)
(0.34304692863149183, 7889.049025804172)
(0.37649358067924676, 8658.286946424892)
(0.4132012400115337, 9502.524621620913)
(0.4534878508128582, 10429.074434257749)
(0.4977023564332111, 11445.961715558402)
(0.5462277217684341, 12561.99425668264)
(0.599484250318941, 13786.83859760083)
(0.657933224657568, 15131.103754038446)
(0.7220809018385465, 16606.433107695313)
(0.7924828983539174, 18225.605255647646)
(0.8697490026177833, 20002.644692442555)
(0.9545484566618341, 21952.943283560868)
(1.0476157527896648, 24093.39358239258)
(1.1497569953977358, 26442.535145452857)
(1.2618568830660204, 29020.715113151688)
(1.384886371393873, 31850.264446989688)
(1.5199110829529336, 34955.69134966455)
(1.6681005372000586, 38363.89354339729)
(1.8307382802953682, 42104.39124513526)
(2.0092330025650473, 46209.582856548695)
(2.2051307399030455, 50715.025583487164)
(2.420128264794382, 55659.74341548585)
(2.6560877829466865, 61086.56513289486)
(2.9150530628251765, 67042.49526928492)
(3.1992671377973836, 73579.1212422323)
(3.511191734215131, 80753.06017884742)
(3.8535285937105295, 88626.44930624917)
(4.229242874389499, 97267.48415449618)
(4.641588833612778, 106751.00923364703)
(5.094138014816378, 117159.16630110046)
(5.590810182512223, 128582.10583420104)
(6.1359072734131725, 141118.7678705477)
(6.7341506577508214, 154877.73897926693)
(7.390722033525779, 169978.19278592928)
(8.11130830789687, 186550.92219747824)
(8.902150854450387, 204739.4722678261)
(9.770099572992255, 224701.38351643435)
(10.722672220103231, 246609.556468932)
(11.768119524349984, 270653.749238749)
(12.91549665014884, 297042.2211211226)
(14.174741629268055, 326003.5364354786)
(15.556761439304715, 357788.54424022336)
(17.073526474706906, 392672.5510672846)
(18.73817422860384, 430957.70549559075)
(20.565123083486515, 472975.6152175048)
(22.5701971963392, 519090.21926603175)
(24.770763559917114, 569700.9402806067)
(27.1858824273294, 625246.1441149232)
(29.836472402833387, 686206.9367522419)
(32.74549162877728, 753111.3314152629)
(35.938136638046274, 826538.8219640672)
(39.44206059437656, 907125.402194706)
(43.287612810830595, 995569.0745132104)
(47.50810162102796, 1.0926358956984943e6)
(52.14008287999685, 1.199166612119739e6)
(57.22367659350217, 1.3160839418791938e6)
(62.80291441834253, 1.444400566954988e6)
(68.92612104349699, 1.5852279045679492e6)
(75.64633275546291, 1.7397857337458825e6)
(83.02175681319744, 1.909412760465983e6)
(91.11627561154891, 2.095578212885587e6)
(100.0, 2.2998945670935796e6)
};
\addlegendentry{order = 6}
\legend{};
\end{loglogaxis}

\end{tikzpicture}
     \caption{Maximum eigenvalues for increasing
    $\tau_s$\label{fig:tau:scaling:max}}
  \end{subfigure}
  \caption{Plot of the minimum and maximum eigenvalue for increasing $\tau_s$
  for a single psuedo-random parameter realization when all four faces are
  Dirichlet for SBP operators with interior orders $2$ (blue line), $4$ (red
  line), and $6$ (brown line).\label{fig:tau:scaling}}
\end{figure}
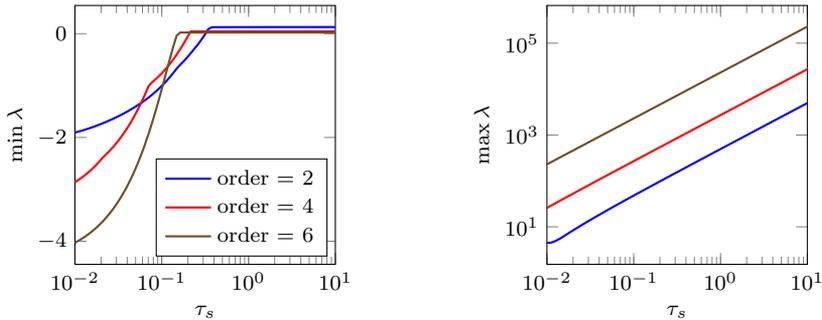
Another question to consider is how the penalty parameter affects the
spectral radius of the operator. In Figure~\ref{fig:tau:scaling} we plot the
minimum and maximum eigenvalues versus increasing $\tau_{s}$; here $\tau_s$ is a
scaling of the penalty parameter so that the actual penalty parameter at each
grid point is $\tau_s$ times the equality version
of~\eref{eqn:tau}. From Figure~\ref{fig:tau:scaling:min} it
is seen that once $\tau_s$ is large enough, the minimum eigenvalue remains
roughly constant. From Figure~\ref{fig:tau:scaling:max} we see that the maximum
eigenvalue increases linearly with $\tau_s$ in all cases, and that the slope of
the line depends on the order of the operators; note that in this figure a
log-log axis has been used so the higher the line the larger the slope.

\begin{figure}
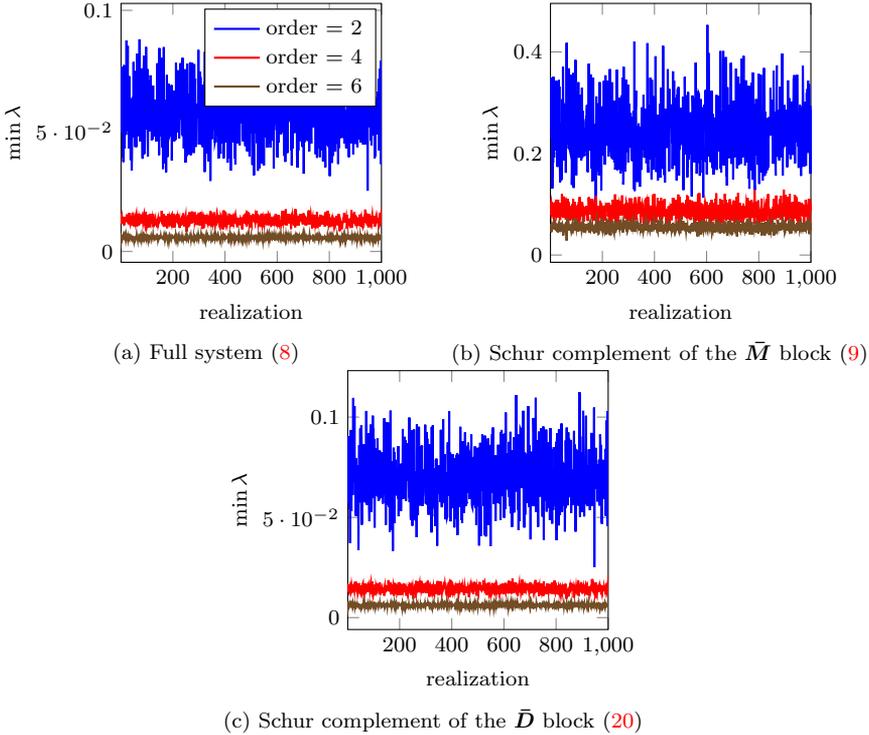

  \centering
  \begin{subfigure}[t]{0.48\textwidth}
    \centering

     \caption{Full system~\eref{eqn:full:system}}
  \end{subfigure}
  \begin{subfigure}[t]{0.48\textwidth}
    \centering
%
     \caption{Schur complement of the $\mm{\bar{M}}$
    block~\eref{eqn:schur:system}}
  \end{subfigure}
  \begin{subfigure}[t]{0.48\textwidth}
    \centering
%
     \caption{Schur complement of the $\mm{\bar{D}}$
    block~\eref{eqn:schur:system:D}}
  \end{subfigure}
  \caption{Plot of the minimum eigenvalue for the full system and two Schur
  complement systems of a two block problem with psuedo-randomly assigned
  coefficient matrix values for SBP operators with interior orders $2$ (blue
  line), $4$ (red line), and $6$ (brown line).\label{fig:min:eig:coupled}}
\end{figure}
We now confirm the positive definiteness of the global problem by considering
two blocks coupled along a single locked interface with Dirichlet boundaries.
Each of the blocks has grids of size $N \times N = (3p-1) \times (3p-1)$ where $2p$
is the interior order of the SBP operator. As before, the coefficient matrix
$\mm{c}$ at each grid point is generated using pseudo-random numbers with the
penalty parameters set to the equality version of~\eref{eqn:tau}. In
Figure~\ref{fig:min:eig:coupled} the minimum eigenvalue for $1000$ realizations
of the material properties is shown.  Eigenvalues from three different systems
are shown: the full system~\eref{eqn:full:system} and the two Schur complement
systems~\eref{eqn:schur:system} and~\eref{eqn:schur:system:D}. In all cases it
is seen that the minimum eigenvalue is positive, confirming that the systems are
positive definite.

\subsection{Numerical Accuracy and Convergence}
Next we explore the accuracy of the method by applying the method of
manufactured solutions (MMS), see for example \citet{Roache}. In the MMS
technique an analytic solution is assumed, and compatible boundary and source
data derived. The domain is taken to be the square $\Omega = \{(x, y) | -2 \leq
x, y \leq 2  \}$. We partition $\Omega$ into the closed unit disk $\Omega_{1} =
\{(x, y) | x^2+y^2\le1 \}$ and $\Omega_2 = \text{cl}(\Omega \setminus
\Omega_{1})$, and define the unit circle $\Gamma_{I} = \{(x, y) | x^2+y^2=1 \}$
to be the interface between $\Omega_{1}$ and $\Omega_{2}$. The domain can be
seen in Figure~\ref{fig:square:circle}. The material properties are taken to be
$\mm{b} = \mm{I}_2$; the metric terms will cause the transformed material
properties $\mm{c}$ to be spatially variable. The right and left boundaries of
$\Omega$ are taken to the Dirichlet, the top and bottom boundaries Neumann, and
the interface $\Gamma_{I}$ will have a jump in the solution.
\begin{figure}
  \begin{center}
\begin{tikzpicture}[]
\begin{axis}[ylabel = {$y$}, xmin = {-2}, xmax = {2}, ymax = {2}, xlabel = {$x$}, width=5cm, height=5cm, ticks=none, ymin = {-2}]\addplot+ [no marks, solid, black]coordinates {
(1.0, 0.0)
(0.5372134, 0.0002884168)
};
\addplot+ [no marks, solid, black]coordinates {
(0.7071067811865476, 0.7071067811865475)
(0.3017369, 0.3000469)
};
\addplot+ [no marks, solid, black]coordinates {
(0.5372134, 0.0002884168)
(0.3017369, 0.3000469)
};
\addplot+ [no marks, solid, black]coordinates {
(6.123233995736766e-17, 1.0)
(0.00333461, 0.5383348)
};
\addplot+ [no marks, solid, black]coordinates {
(0.3017369, 0.3000469)
(0.00333461, 0.5383348)
};
\addplot+ [no marks, solid, black]coordinates {
(-0.7071067811865475, 0.7071067811865476)
(-0.3003842, 0.3066312)
};
\addplot+ [no marks, solid, black]coordinates {
(0.00333461, 0.5383348)
(-0.3003842, 0.3066312)
};
\addplot+ [no marks, solid, black]coordinates {
(-1.0, 1.2246467991473532e-16)
(-0.5360446, -0.0008435451)
};
\addplot+ [no marks, solid, black]coordinates {
(-0.3003842, 0.3066312)
(-0.5360446, -0.0008435451)
};
\addplot+ [no marks, solid, black]coordinates {
(-0.7071067811865475, -0.7071067811865476)
(-0.2957995, -0.3036193)
};
\addplot+ [no marks, solid, black]coordinates {
(-0.5360446, -0.0008435451)
(-0.2957995, -0.3036193)
};
\addplot+ [no marks, solid, black]coordinates {
(-3.8285686989269494e-16, -1.0)
(0.005527262, -0.5353951)
};
\addplot+ [no marks, solid, black]coordinates {
(-0.2957995, -0.3036193)
(0.005527262, -0.5353951)
};
\addplot+ [no marks, solid, black]coordinates {
(0.7071067811865476, -0.7071067811865475)
(0.3022778, -0.2983624)
};
\addplot+ [no marks, solid, black]coordinates {
(0.005527262, -0.5353951)
(0.3022778, -0.2983624)
};
\addplot+ [no marks, solid, black]coordinates {
(0.3022778, -0.2983624)
(0.5372134, 0.0002884168)
};
\addplot+ [no marks, solid, black]coordinates {
(0.3022778, -0.2983624)
(0.003680641, 0.0008290927)
};
\addplot+ [no marks, solid, black]coordinates {
(-0.2957995, -0.3036193)
(0.003680641, 0.0008290927)
};
\addplot+ [no marks, solid, black]coordinates {
(0.003680641, 0.0008290927)
(0.3017369, 0.3000469)
};
\addplot+ [no marks, solid, black]coordinates {
(0.003680641, 0.0008290927)
(-0.3003842, 0.3066312)
};
\addplot+ [no marks, solid, black]coordinates {
(2.0, 2.0)
(2.0, 1.5)
};
\addplot+ [no marks, solid, black]coordinates {
(1.5, 2.0)
(1.459547, 1.459424)
};
\addplot+ [no marks, solid, black]coordinates {
(2.0, 2.0)
(1.5, 2.0)
};
\addplot+ [no marks, solid, black]coordinates {
(2.0, 1.5)
(1.459547, 1.459424)
};
\addplot+ [no marks, solid, black]coordinates {
(1.0, 2.0)
(0.8982224, 1.429057)
};
\addplot+ [no marks, solid, black]coordinates {
(1.5, 2.0)
(1.0, 2.0)
};
\addplot+ [no marks, solid, black]coordinates {
(1.459547, 1.459424)
(0.8982224, 1.429057)
};
\addplot+ [no marks, solid, black]coordinates {
(0.5, 2.0)
(0.3667017, 1.509209)
};
\addplot+ [no marks, solid, black]coordinates {
(1.0, 2.0)
(0.5, 2.0)
};
\addplot+ [no marks, solid, black]coordinates {
(0.8982224, 1.429057)
(0.3667017, 1.509209)
};
\addplot+ [no marks, solid, black]coordinates {
(0.0, 2.0)
(-0.003366247, 1.671689)
};
\addplot+ [no marks, solid, black]coordinates {
(0.5, 2.0)
(0.0, 2.0)
};
\addplot+ [no marks, solid, black]coordinates {
(0.3667017, 1.509209)
(-0.003366247, 1.671689)
};
\addplot+ [no marks, solid, black]coordinates {
(-0.5, 2.0)
(-0.3693048, 1.511127)
};
\addplot+ [no marks, solid, black]coordinates {
(0.0, 2.0)
(-0.5, 2.0)
};
\addplot+ [no marks, solid, black]coordinates {
(-0.003366247, 1.671689)
(-0.3693048, 1.511127)
};
\addplot+ [no marks, solid, black]coordinates {
(-1.0, 2.0)
(-0.8962192, 1.428906)
};
\addplot+ [no marks, solid, black]coordinates {
(-0.5, 2.0)
(-1.0, 2.0)
};
\addplot+ [no marks, solid, black]coordinates {
(-0.3693048, 1.511127)
(-0.8962192, 1.428906)
};
\addplot+ [no marks, solid, black]coordinates {
(-1.5, 2.0)
(-1.460039, 1.460656)
};
\addplot+ [no marks, solid, black]coordinates {
(-1.0, 2.0)
(-1.5, 2.0)
};
\addplot+ [no marks, solid, black]coordinates {
(-0.8962192, 1.428906)
(-1.460039, 1.460656)
};
\addplot+ [no marks, solid, black]coordinates {
(-2.0, 2.0)
(-2.0, 1.5)
};
\addplot+ [no marks, solid, black]coordinates {
(-1.5, 2.0)
(-2.0, 2.0)
};
\addplot+ [no marks, solid, black]coordinates {
(-1.460039, 1.460656)
(-2.0, 1.5)
};
\addplot+ [no marks, solid, black]coordinates {
(-2.0, -2.0)
(-2.0, -1.5)
};
\addplot+ [no marks, solid, black]coordinates {
(-1.5, -2.0)
(-1.460145, -1.459283)
};
\addplot+ [no marks, solid, black]coordinates {
(-2.0, -2.0)
(-1.5, -2.0)
};
\addplot+ [no marks, solid, black]coordinates {
(-2.0, -1.5)
(-1.460145, -1.459283)
};
\addplot+ [no marks, solid, black]coordinates {
(-1.0, -2.0)
(-0.8988851, -1.428805)
};
\addplot+ [no marks, solid, black]coordinates {
(-1.5, -2.0)
(-1.0, -2.0)
};
\addplot+ [no marks, solid, black]coordinates {
(-1.460145, -1.459283)
(-0.8988851, -1.428805)
};
\addplot+ [no marks, solid, black]coordinates {
(-0.5, -2.0)
(-0.3680867, -1.508369)
};
\addplot+ [no marks, solid, black]coordinates {
(-1.0, -2.0)
(-0.5, -2.0)
};
\addplot+ [no marks, solid, black]coordinates {
(-0.8988851, -1.428805)
(-0.3680867, -1.508369)
};
\addplot+ [no marks, solid, black]coordinates {
(0.0, -2.0)
(0.000342336, -1.669245)
};
\addplot+ [no marks, solid, black]coordinates {
(-0.5, -2.0)
(0.0, -2.0)
};
\addplot+ [no marks, solid, black]coordinates {
(-0.3680867, -1.508369)
(0.000342336, -1.669245)
};
\addplot+ [no marks, solid, black]coordinates {
(0.5, -2.0)
(0.3679869, -1.50956)
};
\addplot+ [no marks, solid, black]coordinates {
(0.0, -2.0)
(0.5, -2.0)
};
\addplot+ [no marks, solid, black]coordinates {
(0.000342336, -1.669245)
(0.3679869, -1.50956)
};
\addplot+ [no marks, solid, black]coordinates {
(1.0, -2.0)
(0.8958907, -1.428344)
};
\addplot+ [no marks, solid, black]coordinates {
(0.5, -2.0)
(1.0, -2.0)
};
\addplot+ [no marks, solid, black]coordinates {
(0.3679869, -1.50956)
(0.8958907, -1.428344)
};
\addplot+ [no marks, solid, black]coordinates {
(1.5, -2.0)
(1.459981, -1.46053)
};
\addplot+ [no marks, solid, black]coordinates {
(1.0, -2.0)
(1.5, -2.0)
};
\addplot+ [no marks, solid, black]coordinates {
(0.8958907, -1.428344)
(1.459981, -1.46053)
};
\addplot+ [no marks, solid, black]coordinates {
(2.0, -2.0)
(2.0, -1.5)
};
\addplot+ [no marks, solid, black]coordinates {
(1.5, -2.0)
(2.0, -2.0)
};
\addplot+ [no marks, solid, black]coordinates {
(1.459981, -1.46053)
(2.0, -1.5)
};
\addplot+ [no marks, solid, black]coordinates {
(-2.0, 1.0)
(-1.429238, 0.8987769)
};
\addplot+ [no marks, solid, black]coordinates {
(-2.0, 1.5)
(-2.0, 1.0)
};
\addplot+ [no marks, solid, black]coordinates {
(-1.460039, 1.460656)
(-1.429238, 0.8987769)
};
\addplot+ [no marks, solid, black]coordinates {
(-2.0, 0.5)
(-1.509204, 0.3672318)
};
\addplot+ [no marks, solid, black]coordinates {
(-2.0, 1.0)
(-2.0, 0.5)
};
\addplot+ [no marks, solid, black]coordinates {
(-1.429238, 0.8987769)
(-1.509204, 0.3672318)
};
\addplot+ [no marks, solid, black]coordinates {
(-2.0, 0.0)
(-1.671448, -0.003451758)
};
\addplot+ [no marks, solid, black]coordinates {
(-2.0, 0.5)
(-2.0, 0.0)
};
\addplot+ [no marks, solid, black]coordinates {
(-1.509204, 0.3672318)
(-1.671448, -0.003451758)
};
\addplot+ [no marks, solid, black]coordinates {
(-2.0, -0.5)
(-1.510406, -0.3695582)
};
\addplot+ [no marks, solid, black]coordinates {
(-2.0, 0.0)
(-2.0, -0.5)
};
\addplot+ [no marks, solid, black]coordinates {
(-1.671448, -0.003451758)
(-1.510406, -0.3695582)
};
\addplot+ [no marks, solid, black]coordinates {
(-2.0, -1.0)
(-1.427553, -0.8944508)
};
\addplot+ [no marks, solid, black]coordinates {
(-2.0, -0.5)
(-2.0, -1.0)
};
\addplot+ [no marks, solid, black]coordinates {
(-1.510406, -0.3695582)
(-1.427553, -0.8944508)
};
\addplot+ [no marks, solid, black]coordinates {
(-2.0, -1.0)
(-2.0, -1.5)
};
\addplot+ [no marks, solid, black]coordinates {
(-1.427553, -0.8944508)
(-1.460145, -1.459283)
};
\addplot+ [no marks, solid, black]coordinates {
(2.0, -1.0)
(1.429215, -0.8987309)
};
\addplot+ [no marks, solid, black]coordinates {
(2.0, -1.5)
(2.0, -1.0)
};
\addplot+ [no marks, solid, black]coordinates {
(1.459981, -1.46053)
(1.429215, -0.8987309)
};
\addplot+ [no marks, solid, black]coordinates {
(2.0, -0.5)
(1.509184, -0.3672392)
};
\addplot+ [no marks, solid, black]coordinates {
(2.0, -1.0)
(2.0, -0.5)
};
\addplot+ [no marks, solid, black]coordinates {
(1.429215, -0.8987309)
(1.509184, -0.3672392)
};
\addplot+ [no marks, solid, black]coordinates {
(2.0, 0.0)
(1.671369, 0.003443455)
};
\addplot+ [no marks, solid, black]coordinates {
(2.0, -0.5)
(2.0, 0.0)
};
\addplot+ [no marks, solid, black]coordinates {
(1.509184, -0.3672392)
(1.671369, 0.003443455)
};
\addplot+ [no marks, solid, black]coordinates {
(2.0, 0.5)
(1.510172, 0.3695378)
};
\addplot+ [no marks, solid, black]coordinates {
(2.0, 0.0)
(2.0, 0.5)
};
\addplot+ [no marks, solid, black]coordinates {
(1.671369, 0.003443455)
(1.510172, 0.3695378)
};
\addplot+ [no marks, solid, black]coordinates {
(2.0, 1.0)
(1.426582, 0.8944412)
};
\addplot+ [no marks, solid, black]coordinates {
(2.0, 0.5)
(2.0, 1.0)
};
\addplot+ [no marks, solid, black]coordinates {
(1.510172, 0.3695378)
(1.426582, 0.8944412)
};
\addplot+ [no marks, solid, black]coordinates {
(2.0, 1.0)
(2.0, 1.5)
};
\addplot+ [no marks, solid, black]coordinates {
(1.426582, 0.8944412)
(1.459547, 1.459424)
};
\addplot+ [no marks, solid, black]coordinates {
(-3.8285686989269494e-16, -1.0)
(-0.3680867, -1.508369)
};
\addplot+ [no marks, solid, black]coordinates {
(-0.7071067811865475, -0.7071067811865476)
(-0.8988851, -1.428805)
};
\addplot+ [no marks, solid, black]coordinates {
(-0.7071067811865475, -0.7071067811865476)
(-1.427553, -0.8944508)
};
\addplot+ [no marks, solid, black]coordinates {
(-1.0, 1.2246467991473532e-16)
(-1.510406, -0.3695582)
};
\addplot+ [no marks, solid, black]coordinates {
(-1.0, 1.2246467991473532e-16)
(-1.509204, 0.3672318)
};
\addplot+ [no marks, solid, black]coordinates {
(-0.7071067811865475, 0.7071067811865476)
(-1.429238, 0.8987769)
};
\addplot+ [no marks, solid, black]coordinates {
(-0.7071067811865475, 0.7071067811865476)
(-0.8962192, 1.428906)
};
\addplot+ [no marks, solid, black]coordinates {
(6.123233995736766e-17, 1.0)
(-0.3693048, 1.511127)
};
\addplot+ [no marks, solid, black]coordinates {
(6.123233995736766e-17, 1.0)
(0.3667017, 1.509209)
};
\addplot+ [no marks, solid, black]coordinates {
(0.7071067811865476, 0.7071067811865475)
(0.8982224, 1.429057)
};
\addplot+ [no marks, solid, black]coordinates {
(0.7071067811865476, 0.7071067811865475)
(1.426582, 0.8944412)
};
\addplot+ [no marks, solid, black]coordinates {
(1.0, 0.0)
(1.510172, 0.3695378)
};
\addplot+ [no marks, solid, black]coordinates {
(1.0, 0.0)
(1.509184, -0.3672392)
};
\addplot+ [no marks, solid, black]coordinates {
(0.7071067811865476, -0.7071067811865475)
(1.429215, -0.8987309)
};
\addplot+ [no marks, solid, black]coordinates {
(0.7071067811865476, -0.7071067811865475)
(0.8958907, -1.428344)
};
\addplot+ [no marks, solid, black]coordinates {
(0.3679869, -1.50956)
(-3.8285686989269494e-16, -1.0)
};
\draw[very thick, red] (axis cs:0, 0) circle[radius=1];
\end{axis}

\end{tikzpicture}
     \caption{Domain used for MMS solution~\eref{eqn:mms}. The thick red line is
    the interface between the two subdomains $\Omega_{1}$ and $\Omega_{2}$. The
    thin black lines show the finite difference block
    interfaces.\label{fig:square:circle}}
  \end{center}
\end{figure}
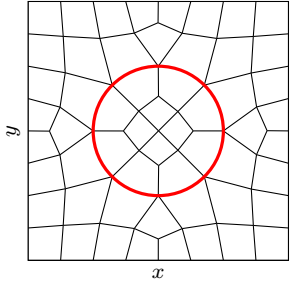

The manufactured solution is taken to be
\begin{equation}\label{eqn:mms}
  u(x,y) =
  \begin{cases}
    \frac{e}{1+e}\left( 2 - e^{-r^2}\right) r \sin(\theta), \quad (x, y) \in \Omega_1,\\
    {(r- 1)}^2 \cos(\theta) + (r - 1)\sin(\theta),\quad (x, y) \in \Omega_2,
  \end{cases}
\end{equation}
where $r = \sqrt{x^2 + y^2}$ and $-\pi \le \theta = \tan^{-1}(y/x) < \pi$. This
solution has the property that along $\Gamma_{I}$ the solution $u$ is
discontinuous but the weighted normal derivative $\vv{n} \cdot \nabla{u}$ is
continuous. The boundary, jump, and forcing data are found by
using~\eref{eqn:mms} in governing equations~\eref{eqn:gov}.

The test is run with the block decomposition shown in
Figure~\ref{fig:square:circle}. Each block uses an $(N+1) \times (N+1)$ grid of
points where $N$ will be increased with grid refinement. The error is measured
using the discrete norm
\begin{equation*}
  \begin{split}
    \text{error}_N &= \sqrt{\sum_{b=1}^{N_{b}} \VV{\Delta}_b^{T} \MM{J}_{b} \left(\mm{H}
    \otimes \mm{H}\right) \VV{\Delta}_{b}},\\
    \VV{\Delta}_b &= \VV{u}_{b} - u\left(\VV{x}_{b}, \VV{y}_{b}\right).
  \end{split}
\end{equation*}
Here $\VV{J}_{b}$ is the diagonal matrix of Jacobian determinants for block $b$
and $u\left(\VV{x}_{b}, \VV{y}_{b}\right)$ is the exact solution~\eref{eqn:mms}
evaluated at the grid points of block $b$.  Table~\ref{tab:error} shows the
error and convergence rate estimates with increasing $N$ for $2p = 2, 4, 6$, and
reflect global convergence rates of 2, 4, and 5, respectively.
\begin{table}
  \begin{center}
    \begin{tabular}{ccccccc}
      \toprule
      \multicolumn{1}{c}{} &  \multicolumn{2}{c}{2nd Order} &    \multicolumn{2}{c}{4th Order}&    \multicolumn{2}{c}{6th Order} \\
      \midrule
      $N$              &            error$_N$  &   rate &              error$_N$ &    rate &              error$_N$ &   rate \\
      $17 \times 2^0 $ & $2.90 \times 10^{-4}$ &        & $1.81 \times 10^{-6}$  &         & $3.02 \times 10^{-7} $ &        \\
      $17 \times 2^1 $ & $7.23 \times 10^{-5}$ & $2.01$ & $1.25 \times 10^{-7}$  & $3.86$  & $1.10 \times 10^{-8} $ & $4.66$ \\
      $17 \times 2^2 $ & $1.80 \times 10^{-5}$ & $2.00$ & $8.32 \times 10^{-9}$  & $3.91$  & $4.26 \times 10^{-10}$ & $4.81$ \\
      $17 \times 2^3 $ & $4.51 \times 10^{-6}$ & $2.00$ & $5.45 \times 10^{-10}$ & $3.93$  & $1.42 \times 10^{-11}$ & $4.90$ \\
      \bottomrule
    \end{tabular}
  \end{center}
  \caption{Error and convergence rates using the method of manufactured
  solutions.\label{tab:error}}
\end{table}

As a final numerical result, the accuracy of the weighted normal derivative
along the interface is considered. Using the same problem setup as above, the
weighted interface derivative are taken to be the penalty term
$\vv{\hat{\sigma}}$ computed using \eref{eqn:IP:flux}; by construction
\eref{eqn:locked:IC} implies that the normal derivative is equal and magnitude
and opposite in sign across the interface. The error in the normal derivative is
defined to be
\begin{equation*}
  \begin{split}
    \text{interface error}_N &= \sqrt{\sum_{f \in \FF_{I}} \vv{\Delta}_f^{T}
    \mm{S}_{J,f} \mm{H} \vv{\Delta}_{f}},\\
    \VV{\Delta}_f &= \vv{\hat{\sigma}}_{f} - \sigma\left(\vv{x}_{f},
    \vv{y}_{f}\right).
  \end{split}
\end{equation*}
The results of this are show in Table~\ref{tab:face:error}.
\begin{table}
  \begin{center}
    \begin{tabular}{ccccccc}
      \toprule
      \multicolumn{1}{c}{} &  \multicolumn{2}{c}{2nd Order} &    \multicolumn{2}{c}{4th Order}&    \multicolumn{2}{c}{6th Order} \\
      \midrule
      $N$              &   interface error$_N$ &   rate &   interface error$_N$ &    rate &  interface error$_N$ &   rate \\
      $17 \times 2^0 $ & $4.93 \times 10^{-3}$ &        & $1.35 \times 10^{-4}$ &        & $2.39 \times 10^{-5}$ &        \\
      $17 \times 2^1 $ & $1.83 \times 10^{-3}$ & $1.43$ & $2.69 \times 10^{-5}$ & $2.33$ & $2.53 \times 10^{-6}$ & $3.24$ \\
      $17 \times 2^2 $ & $6.66 \times 10^{-4}$ & $1.46$ & $5.03 \times 10^{-6}$ & $2.42$ & $2.46 \times 10^{-7}$ & $3.37$ \\
      $17 \times 2^3 $ & $2.39 \times 10^{-4}$ & $1.48$ & $9.16 \times 10^{-7}$ & $2.46$ & $2.28 \times 10^{-8}$ & $3.43$ \\
      \bottomrule
    \end{tabular}
  \end{center}
  \caption{Error and convergence rates using the method of manufactured
  solutions for the interface normal derivative.\label{tab:face:error}}
\end{table}
As can be seen, the interface derivative converges at a rate of the $p + 1 / 2$.
Since the boundary derivative operators only have accuracy of $p$ a reduced
convergence rate is expected.

\begin{figure}
  \begin{subfigure}[t]{0.5\textwidth}
    \centering
    \includegraphics{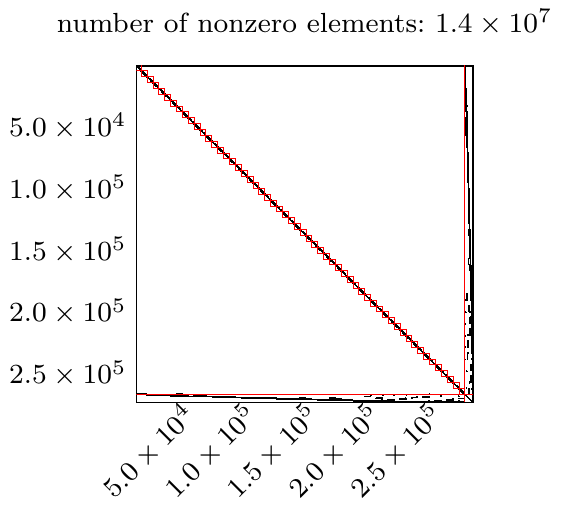}
    \caption{Monolithic system \eref{eqn:full:system}}
    \label{fig:spy:monolithic}
  \end{subfigure}
  \begin{subfigure}[t]{0.5\textwidth}
    \centering
    \includegraphics{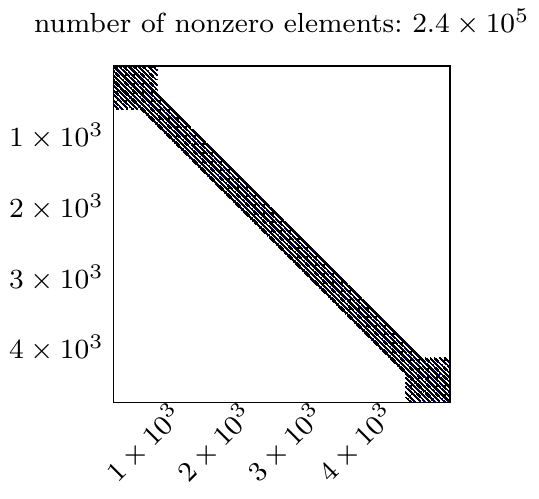}
    \caption{Single finite difference block}
    \label{fig:spy:single}
  \end{subfigure}
  \begin{subfigure}[t]{0.5\textwidth}
    \centering
    \includegraphics{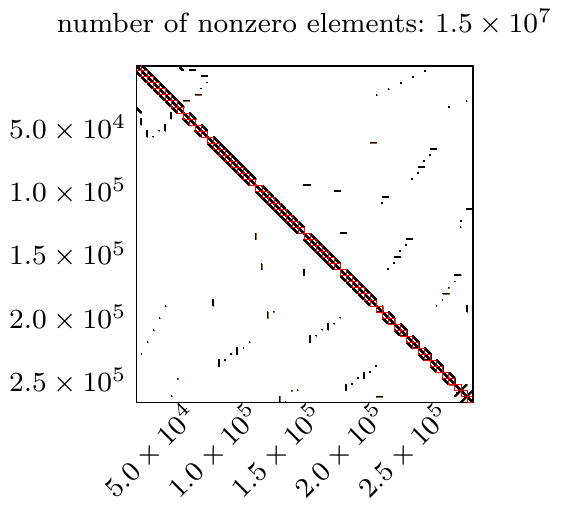}
    \caption{Volume system \eref{eqn:schur:system:D}}
    \label{fig:spy:volume}
  \end{subfigure}
  \begin{subfigure}[t]{0.5\textwidth}
    \centering
    \includegraphics{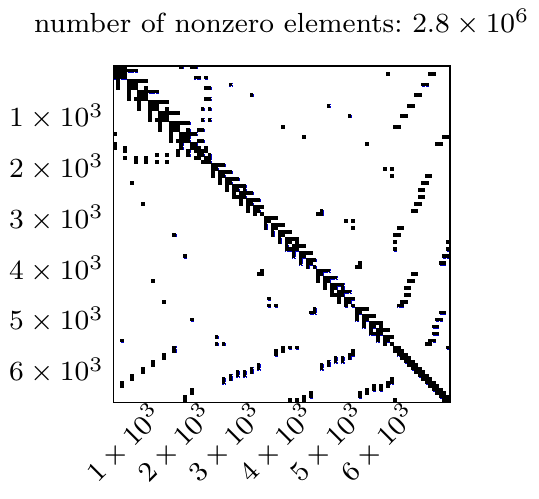}
    \caption{Trace system \eref{eqn:schur:system}}
    \label{fig:spy:trace}
  \end{subfigure}
  \caption{Spy plots showing the sparsity pattern for various systems related to
  the mesh in Figure~\ref{fig:square:circle} with SBP order $6$ and $N = 17 \times
  2^2$. The red lines in Subfigures~\ref{fig:spy:monolithic}
  and~\ref{fig:spy:volume} denote the diagonal submatrices associated with a
  single finite difference block and for the monolithic system the connections
  between the trace and volume variables, $\Bar{F}$ in \eref{eqn:full:system}.
  \label{fig:spy}}
\end{figure}
To highlight the sparsity and reduction of system size we consider spy plots in
Figure~\ref{fig:spy} for the following matrices: the four matrices in the
monolithic system with both volume and trace variables \eref{eqn:full:system},
a single finite difference block from the monolithic system, the Schur
complement matrix obtained by removing the volume variables
\eref{eqn:schur:system}, and the Schur complement matrix obtained by removing
the trace variables (20).
\begin{table}
  \begin{center}
    \begin{tabular}{rrrr}
      \toprule
      \multicolumn{1}{c}{$N$} &
        $N_{p}^{\text{(vol)}}$ &
        $N_{p}^{\text{(tr)}}$ &
        $N_{p}^{\text{(vol)}} / N_{p}^{\text{(tr)}}$\\
      \midrule
      $17 \times 2^0$ &   $18144$ &  $1728$ & $10.5$ \\
      $17 \times 2^1$ &   $68600$ &  $3360$ & $20.4$ \\
      $17 \times 2^2$ &  $266616$ &  $6624$ & $40.3$ \\
      $17 \times 2^3$ & $1051064$ & $13152$ & $79.9$ \\
      \bottomrule
    \end{tabular}
  \end{center}
  \caption{Comparison of the number of volume and trace points for the mesh
  shown in Figure~\ref{fig:square:circle} with the mesh sizes of
  Table~\ref{tab:error}.\label{tab:Np}}
\end{table}
Additionally, Table~\ref{tab:Np} gives the number of volume and trace
points for each $N$ are given.  If $N_{b}$ is the number of blocks and $N_{I}$
the number of internal interfaces, the number of volume and trace points are
\begin{align*}
  N_{p}^{\text{(vol)}} &= {(N+1)}^2 N_{b},\\
  N_{p}^{\text{(tr)}}  &= (N+1) N_{I},
\end{align*}
respectively; the mesh in Figure~\ref{fig:square:circle} has $N_b = 56$
blocks and $N_{I} = 96$ internal interfaces.

\section{Conclusions}\label{sec:conclusions}

We have developed a hybridized, summation-by-parts finite difference method for
elliptic PDEs, where boundary and interface conditions are enforced weakly
through the simultaneous-approximation-term method. The hybridization defines a
global and local problem, which through the Schur complement, results in a
linear system with reduced size.  We proved positive-definiteness of both the
local and global problems with arbitrarily heterogeneous material properties.
The theoretical results were corroborated through numerical experiments and
showed convergence to an exact solution at the expected rate.

All of the results in \sref{sec:numerical} used sparse direct solves; the sparse
Cholesky factorization within Julia is CHOLMOD \citep{cholmod}. Even though the
trace variable system is smaller and has fewer non-zero entries, it is still
possible that this system may be harder to solve in general than either the
volume variable system or the monolithic system. If direct methods are used,
comparisons with other direct methods, such as nested dissection, e.g.,
\citep{george1973nested, davis2006direct}, should be considered. Similarly, if
iterative methods are to be used, there is the possibility of using direct
methods for the local problems and iterative methods for the trace system. To do
this efficiently will require the development of preconditioners for the trace
system. Since the discrete scheme closely resembles the hybridized discontinuous
Galerkin interior penalty method
\citep[IP-H]{CockburnGopalakrishnanLazarov2009}, problems which can be handled
with IP-H maybe amenable to the presented hybridized SBP scheme.  Iterative
methods could also be explored for the local problem, such as SBP-SAT specific
geometric multigrid techniques \citep{ruggiu2018new}.

Another avenue for future work is construction of hybridized SBP schemes that
more closely resemble other hybridizable discontinuous Galerkin methods. The
scheme may also be applicable to nonlinear elliptic problems, the challenge here
would be efficient methods for solving the local problems, as in the non-linear
context direct solves may be inefficient.  In recent years non-conforming
coupling of SBP schemes has been of particular interest
\citep{mattsson2010stable, wang2016high, kozdon2016stable}, and could be
explored in the hybridized context. One challenge which may arise is the need
for a single trace variable, which maybe challenging for most of the present
non-conforming SBP formulations.

\appendix
\section{Proofs of Key Results}
To simplify the presentation of the results, the proofs of the key results in
the paper are given here in the appendix.
\subsection{Proof of Theorem~\ref{thm:loc:PD} (Symmetric Positive Definiteness of the
Local Problem)}\label{sec:app:loc:PD}
Here we provide conditions that ensure that the local problem is symmetric
positive definite. To do this we need a few auxiliary lemmas.

First we assume that the operators $\MM{A}_{rr}^{\left(c_{rr}\right)}$ and
$\MM{A}_{ss}^{\left(c_{ss}\right)}$ are compatible with the first derivative
(volume) operator $\mm{D}$ in the sense of \citet[Definition 2.4]{Mat12}:
\begin{assumption}[Remainder Assumption]\label{assmp:remainder}
  The matrices $\MM{A}_{rr}^{\left(c_{rr}\right)}$ and
  $\MM{A}_{ss}^{\left(c_{ss}\right)}$ satisfy the following remainder
  equalities:
  \begin{equation*}
    \begin{split}
      \MM{A}_{rr}^{\left(c_{rr}\right)}& =
      \left(\mm{I} \otimes \mm{D}^{T}\right)
      \MM{C}_{rr} \left(\mm{H} \otimes \mm{H}\right)
      \left(\mm{I} \otimes \mm{D}\right)
      +
      \MM{R}_{rr}^{\left(c_{rr}\right)},\\
      \MM{A}_{ss}^{\left(c_{ss}\right)}& =
      \left(\mm{D}^{T} \otimes \mm{I}\right)
      \MM{C}_{ss} \left(\mm{H} \otimes \mm{H}\right)
      \left(\mm{D} \otimes \mm{I}\right)
      +
      \MM{R}_{ss}^{\left(c_{ss}\right)},
    \end{split}
  \end{equation*}
  where $\MM{R}_{rr}^{\left(c_{rr}\right)}$ and
  $\MM{R}_{ss}^{\left(c_{ss}\right)}$ are symmetric positive semidefinite
  matrices and that
  \begin{align*}
    \VV{1} \in \nullspace\left(\MM{A}_{rr}^{\left(c_{rr}\right)}\right),
    \qquad
    \VV{1} \in \nullspace\left(\MM{A}_{ss}^{\left(c_{ss}\right)}\right).
  \end{align*}
\end{assumption}
The assumption on the nullspace was not a part of the original assumption of
from \citet{Mat12}, but is reasonable for a consistent approximation of the
second derivative.  The operators used in \sref{sec:numerical} satisfy the
Remainder Assumption~\citep{Mat12}.

We also utilize the following lemma from~\citet[Lemma 3]{VM14} which relates the
$\MM{A}_{rr}^{\left(c_{rr}\right)}$ and $\MM{A}_{ss}^{\left(c_{ss}\right)}$ to
boundary derivative operators $\vv{d}_{0}$ and $\vv{d}_{N}$:
\begin{lemma}[Borrowing Lemma]\label{lemma:borrowing}
  The matrices $\MM{A}_{rr}^{\left(c_{rr}\right)}$ and
  $\MM{A}_{ss}^{\left(c_{ss}\right)}$ satisfy the following borrowing
  equalities:
  \begin{equation*}
    \begin{split}
    \MM{A}_{rr}^{\left(c_{rr}\right)}& =
    h\beta
    \left(\mm{I} \otimes \vv{d}_{0}\right)
    \mm{H}\mm{\mathcal{C}}_{rr}^{0:}
    \left(\mm{I} \otimes \vv{d}_{0}^{T}\right)\\
    & \quad +
    h\beta
    \left(\mm{I} \otimes \vv{d}_{N}\right)
    \mm{H}\mm{\mathcal{C}}_{rr}^{N:}
    \left(\mm{I} \otimes \vv{d}_{N}^{T}\right)
    +\MM{\mathcal{A}}_{rr}^{\left(c_{rr}\right)},\\
    \MM{A}_{ss}^{\left(c_{ss}\right)}& =
    h\beta
    \left(\vv{d}_{0} \otimes \mm{I}\right)
    \mm{H}\mm{\mathcal{C}}_{ss}^{:0}
    \left(\vv{d}_{0}^{T} \otimes \mm{I}\right)\\
    & \quad +
    h\beta
    \left(\vv{d}_{N} \otimes \mm{I}\right)
    \mm{H}\mm{\mathcal{C}}_{ss}^{:N}
    \left(\vv{d}_{N}^{T} \otimes \mm{I}\right)
    +\MM{\mathcal{A}}_{ss}^{\left(c_{ss}\right)}.
    \end{split}
  \end{equation*}
  Here $\MM{\mathcal{A}}_{rr}^{\left(c_{rr}\right)}$ and
  $\MM{\mathcal{A}}_{ss}^{\left(c_{ss}\right)}$ are symmetric positive
  semidefinite matrices and the parameter $\beta$ depends on the order of
  the operators but is independent of $N$. The diagonal matrices
  $\mm{\mathcal{C}}_{rr}^{0:}$, $\mm{\mathcal{C}}_{rr}^{N:}$,
  $\mm{\mathcal{C}}_{ss}^{:0}$, and $\mm{\mathcal{C}}_{ss}^{:N}$
  have nonzero elements:
  \begin{equation}
    \label{eqn:borrow:C}
    \begin{alignedat}{2}
      {\left[\mm{\mathcal{C}}_{rr}^{0:}\right]}_{jj} &=
      \min_{k=0,\dots,l} {\left\{c_{rr}\right\}}_{kj},\quad&
      {\left[\mm{\mathcal{C}}_{rr}^{N:}\right]}_{jj} &=
      \min_{k=N-l,\dots,N} {\left\{c_{rr}\right\}}_{kj},\\
      {\left[\mm{\mathcal{C}}_{ss}^{:0}\right]}_{ii} &=
      \min_{k=0,\dots,l} {\left\{c_{ss}\right\}}_{ik},&
      {\left[\mm{\mathcal{C}}_{ss}^{:N}\right]}_{ii} &=
      \min_{k=N-l,\dots,N} {\left\{c_{ss}\right\}}_{ik},
    \end{alignedat}
  \end{equation}
  where $l$ is a parameter that depends on the order of the scheme and the
  notation ${\{\cdot\}}_{ij}$ denotes that the grid function inside the bracket
  is evaluated at grid point $i,j$.
\end{lemma}
The values of $\beta$ and $l$ used in
the~\nameref{lemma:borrowing} (Lemma~\ref{lemma:borrowing}) for the operators
used in this work are given in Table~\ref{tab:borrow:params}.
\begin{table}
  \small
  \begin{center}
    \begin{tabular}{cccc}
      \toprule
              & 2nd Order   & 4th Order    & 6th Order\footnoteref{footnote:x1}\\
      \midrule
      $l$     & 2           & 4            & 6           \\
      $\beta$ & 0.363636363 & 0.2505765857 & 0.1878687080\\
      \bottomrule
    \end{tabular}
    \caption{\nameref{lemma:borrowing} parameters $l$ and $\beta$ for operators
    used in this work~\citep[Table 1]{VM14}.\label{tab:borrow:params}}
  \end{center}
\end{table}

We additionally make the following linearity assumption (which the operators we
use satisfy) concerning the operators's dependence on the variable coefficients
and an assumption concerning the symmetric positive definiteness of the variable
coefficient matrix at each grid point.
\begin{assumption}\label{assumption:decomp}
  The matrices $\MM{A}_{rr}^{(c_{rr})}$ and $\MM{A}_{ss}^{(c_{ss})}$ depend
  linearly on the coefficient grid functions $c_{rr}$ and $c_{ss}$ so that they
  can be decomposed as
  \begin{equation*}
    \begin{split}
      \MM{A}_{rr}^{(c_{rr})} &= \MM{A}_{rr}^{(c_{rr}-\delta)} +
      \MM{A}_{rr}^{(\delta)},\\
      \MM{A}_{ss}^{(c_{ss})} &= \MM{A}_{ss}^{(c_{ss}-\delta)} +
      \MM{A}_{ss}^{(\delta)},
    \end{split}
  \end{equation*}
  where $\delta$ is a grid function.
\end{assumption}
\begin{assumption}\label{assumption:coeff}
  At every grid point the grid functions $c_{rr}$, $c_{ss}$, and $c_{rs} =
  c_{sr}$ satisfy
  \begin{align*}
    c_{rr} &> 0, &
    c_{ss} &> 0, &
    c_{rr} c_{ss} > c_{rs}^2
  \end{align*}
  which implies that the matrix
  \begin{align*}
    C =
    \begin{bmatrix}
      c_{rr} & c_{rs}\\
      c_{rs} & c_{ss}
    \end{bmatrix}
  \end{align*}
  is symmetric positive definite with eigenvalues
  \begin{equation}
    \label{eqn:lambda}
    \begin{split}
      \psi_{\max} = \frac{1}{2}\left(c_{rr} + c_{ss} +
      \sqrt{{\left(c_{rr}-c_{ss}\right)}^{2} + 4c_{rs}^{2}}\right),\\
      \psi_{\min} = \frac{1}{2}\left(c_{rr} + c_{ss} -
      \sqrt{{\left(c_{rr}-c_{ss}\right)}^{2} + 4c_{rs}^{2}}\right).
    \end{split}
  \end{equation}
\end{assumption}

We now state the following lemma which allows us to separate $\MM{A}$ into
three symmetric positive definite matrices by peeling off $\psi_{\min}$ at
every grid point.
\begin{lemma}\label{app:lemma:split}
  The matrix $\MM{A}$, defined by \eref{eqn:MMA}, can be written in the form
  \begin{align*}
    \MM{A} = \MM{\mathcal{A}} + \MM{A}_{rr}^{(\psi_{\min})}
    +
    \MM{A}_{ss}^{(\psi_{\min})},
  \end{align*}
  where $\MM{\mathcal{A}}$, $\MM{A}_{rr}^{(\psi_{\min})}$, and
  $\MM{A}_{ss}^{(\psi_{\min})}$ are symmetric positive semidefinite
  matrices. Here $\psi_{\min}$ is the grid function defined
  by~\eref{eqn:lambda}. Furthermore the nullspace of $\MM{\mathcal{A}}$ is
  $\nullspace(\MM{\mathcal{A}}) = \spanspace\{\VV{1}\}$, where
  $\VV{1}$ is the vector of ones.
\end{lemma}
\begin{proof}
  By Assumption~\ref{assumption:decomp} we can write
  \begin{align*}
    \MM{A} &= \MM{A}^{\left(c_{rr}-\psi_{\min}\right)}_{rr}
    + \MM{A}^{\left(c_{ss}-\psi_{\min}\right)}_{ss}
    + \MM{A}^{\left(c_{rs}\right)}_{rs}
    + \MM{A}^{\left(c_{sr}\right)}_{sr}
    + \MM{A}^{\left(\psi_{\min}\right)}_{rr}
    + \MM{A}^{\left(\psi_{\min}\right)}_{ss}.
  \end{align*}
  The matrix
  \begin{align*}
    \MM{\mathcal{A}} &= \MM{A}^{\left(c_{rr}-\psi_{\min}\right)}_{rr}
    + \MM{A}^{\left(c_{ss}-\psi_{\min}\right)}_{ss}
    + \MM{A}^{\left(c_{rs}\right)}_{rs}
    + \MM{A}^{\left(c_{sr}\right)}_{sr}
  \end{align*}
  is clearly symmetric by construction. To show that the matrix is positive
  semidefinite we note that
  \begin{align}
    \label{eqn:proof:rr}
    \VV{u}^{T} \MM{A}^{\left(c_{rr}-\psi_{\min}\right)}_{rr}\VV{u}
    &\ge \VV{u}_{r}^{T} \left(\mm{H} \otimes \mm{H}\right) \left(\MM{C}_{rr}-\MM{\psi}_{\min}\right) \VV{u}_{r},\\
    \label{eqn:proof:ss}
    \VV{u}^{T} \MM{A}^{\left(c_{ss}-\psi_{\min}\right)}_{ss}\VV{u}
    &\ge \VV{u}_{s}^{T} \left(\mm{H} \otimes \mm{H}\right) \left(\MM{C}_{ss}-\MM{\psi}_{\min}\right) \VV{u}_{s},\\
    \label{eqn:proof:rs}
    \VV{u}^{T} \MM{A}^{\left(c_{rs}\right)}_{rs}\VV{u}
    &=
    \VV{u}^{T} \MM{A}^{\left(c_{sr}\right)}_{sr}\VV{u}
    =\VV{u}_{r}^{T} \left(\mm{H} \otimes \mm{H}\right) \MM{C}_{rs} \VV{u}_{s}.
  \end{align}
  Here we have defined the vectors $\VV{u}_{r} = \left(\mm{I} \otimes
  \mm{D}\right) \VV{u}$ and $\VV{u}_{s} = \left(\mm{D} \otimes
  \mm{I}\right) \VV{u}$. Inequalities~\eref{eqn:proof:rr}
  and~\eref{eqn:proof:ss} follow from the \nameref{assmp:remainder}
  and equality~\eref{eqn:proof:rs} follows from \eref{eqn:A}
  and the symmetry assumption ($c_{rs} = c_{sr}$). Using
  relationships~\eref{eqn:proof:rr}--\eref{eqn:proof:rs} we have that
  \begin{align}
    \label{eqn:umathcalAu}
    \VV{u}^{T}\MM{\mathcal{A}}\VV{u} &\ge
    \sum_{i=0}^{N} \sum_{j=0}^{N}
    {\left\{\left(\mm{H} \otimes \mm{H}\right)\right\}}_{ij}
    {\left\{
    \begin{bmatrix}
      u_{r}\\
      u_{s}
    \end{bmatrix}^{T}
    \begin{bmatrix}
      c_{rr} - \psi_{\min} & c_{rs}\\
      c_{rs} & c_{ss} - \psi_{\min}
    \end{bmatrix}
    \begin{bmatrix}
      u_{r}\\
      u_{s}
    \end{bmatrix}
    \right\}}_{i,j},
  \end{align}
  where the notation ${\left\{\cdot\right\}}_{i,j}$ denotes that the grid
  function inside the brackets is evaluated at grid point $i,j$. The $2\times2$
  matrix in \eref{eqn:umathcalAu} is the shift of the matrix $C$ by its minimum
  eigenvalue, thus by Assumption~\ref{assumption:coeff} is symmetric positive
  semidefinite.  It then follows that each term in the summation is non-negative
  and the matrix $\MM{\mathcal{A}}$ is symmetric positive semidefinite.

  The matrices $\MM{A}_{rr}^{(\psi_{\min})}$ and
  $\MM{A}_{ss}^{(\psi_{\min})}$ are clearly symmetric by construction, with
  positive semidefiniteness following from the positivity of $\psi_{\min}$
  and the \nameref{assmp:remainder}.

  We now show that $\nullspace(\MM{\mathcal{A}}) = \spanspace\{\VV{1}\}$.
  For the right-hand side of \eref{eqn:umathcalAu} to be zero it is required
  that ${\left(u_{r}\right)}_{i,j} = {\left(u_{s}\right)}_{i,j} = 0$ for all
  $i,j$. The only way for this to happen is if $\VV{u} = \alpha \VV{1}$ for some
  constant $\alpha$. Thus we have shown that $\nullspace(\MM{\mathcal{A}})
  \subseteq \spanspace\{\VV{1}\}$. To show equality we note that
  by Assumption~\ref{assmp:remainder} and the structure of
  $\MM{A}_{rs}^{(C_{rs})}$ and $\MM{A}_{sr}^{(C_{sr})}$ given in
  \eref{eqn:A}, the constant vector $\VV{1} \in
  \nullspace(\MM{\mathcal{A}})$. Together the above two results imply that
  $\nullspace(\MM{\mathcal{A}}) = \spanspace\{\VV{1}\}$.
\end{proof}

We now state the following lemma concerning $\MM{A}_{rr}^{(\psi_{\min})}$ and
$\MM{A}_{ss}^{(\psi_{\min})}$ which combine the \nameref{assmp:remainder} and
the \nameref{lemma:borrowing} to provide terms that can be used to bound indefinite
terms in the local operator $\MM{M}$.
\begin{lemma}\label{app:lemma:bound:deriv}
    The matrices $\MM{A}_{rr}^{(\psi_{\min})}$ and
    $\MM{A}_{ss}^{(\psi_{\min})}$ satisfy the following inequalities:
    \begin{equation*}
      \begin{split}
        \VV{u}^{T} \MM{A}_{rr}^{(\psi_{\min})} \VV{u}& \ge
      \frac{1}{2}\left[
        h\beta
      {\left(\vv{v}_{r}^{0:}\right)}^{T}
      \mm{H}\mm{\Psi}_{\min}^{0:}
      \vv{v}_{r}^{0:}
      +
        h\beta
      {\left(\vv{v}_{r}^{N:}\right)}^{T}
      \mm{H}\mm{\Psi}_{\min}^{N:}
      \vv{v}_{r}^{N:}
      \right]\\
      & \quad +
      \frac{1}{2}\left[
      h \alpha
      {\left(\vv{w}_{r}^{:0}\right)}^{T}
      \mm{H}\mm{\Psi}_{\min}^{:0}
      \vv{w}_{r}^{:0}
      +
      h \alpha
      {\left(\vv{w}_{r}^{:N}\right)}^{T}
      \mm{H}\mm{\Psi}_{\min}^{:N}
      \vv{w}_{r}^{:N}
      \right],\\
        \VV{u}^{T} \MM{A}_{ss}^{(\Psi_{\min})} \VV{u}& \ge
      \frac{1}{2}\left[
      h \alpha
      {\left(\vv{w}_{s}^{0:}\right)}^{T}
      \mm{H}\mm{\Psi}_{\min}^{0:}
      \vv{w}_{s}^{0:}
      +
      h \alpha
      {\left(\vv{w}_{s}^{N:}\right)}^{T}
      \mm{H}\mm{\Psi}_{\min}^{N:}
      \vv{w}_{s}^{N:}
      \right]
    \\
      & \quad +
      \frac{1}{2}\left[
      h\beta
      {\left(\vv{v}_{s}^{:0}\right)}^{T}
      \mm{H}\mm{\Psi}_{\min}^{:0}
      \vv{v}_{s}^{:0}
      +
        h\beta
      {\left(\vv{v}_{s}^{:N}\right)}^{T}
      \mm{H}\mm{\Psi}_{\min}^{:N}
      \vv{v}_{s}^{:N}
      \right],
      \end{split}
    \end{equation*}
    with
    $\alpha = \min \left\{{\left\{\mm{H}\right\}}_{00},
                          {\left\{\mm{H}\right\}}_{NN}\right\} / h$,
    i.e., the unscaled corner value in the H-matrix, and the (boundary)
    derivative vectors are defined as
    \begin{equation*}
    \begin{alignedat}{2}
      \vv{v}_{r}^{0:} &= \left(\mm{I} \otimes \vv{d}_{0}^{T}\right) \VV{u},\quad&
      \vv{v}_{r}^{N:} &= \left(\mm{I} \otimes \vv{d}_{N}^{T}\right) \VV{u},\\
      \vv{w}_{r}^{:0} &= \left(\vv{e}_{0}^{T} \otimes \mm{D}\right) \VV{u},&
      \vv{w}_{r}^{:N} &= \left(\vv{e}_{N}^{T} \otimes \mm{D}\right) \VV{u},\\
      \vv{w}_{s}^{0:} &= \left(\mm{D} \otimes \vv{e}_{0}^{T}\right) \VV{u},&
      \vv{w}_{s}^{N:} &= \left(\mm{D} \otimes \vv{e}_{N}^{T}\right) \VV{u},\\
      \vv{v}_{s}^{:0} &= \left(\vv{d}_{0}^{T} \otimes \mm{I}\right) \VV{u},&
      \vv{v}_{s}^{:N} &= \left(\vv{d}_{N}^{T} \otimes \mm{I}\right) \VV{u}.
    \end{alignedat}
    \end{equation*}
    The diagonal matrices $\MM{\Psi}_{\min}^{0:}$, $\MM{\Psi}_{\min}^{N:}$,
    $\MM{\Psi}_{\min}^{:0}$, and $\MM{\Psi}_{\min}^{:N}$ are defined
    by~\eref{eqn:borrow:C} using $\psi_{\min}$.
\end{lemma}
\begin{proof}
  We will prove the relationship for $\MM{A}_{rr}^{(\psi_{\min})}$, and the
  proof $\MM{A}_{ss}^{(\psi_{\min})}$ is analogous. First we note that by the
  \nameref{lemma:borrowing} it immediately follows that
  \begin{equation}
    \label{eqn:proof:borrow:first}
    \VV{u}^{T} \MM{A}_{rr}^{(\psi_{\min})} \VV{u} \ge
        h\beta
      {\left(\vv{v}_{r}^{0:}\right)}^{T}
      \mm{H}\mm{\Psi}_{\min}^{0:}
      \vv{v}_{r}^{0:}
      +
      h\beta
      {\left(\vv{v}_{r}^{N:}\right)}^{T}
      \mm{H}\mm{\Psi}_{\min}^{N:}
      \vv{v}_{r}^{N:}.
  \end{equation}
  Additionally by the \nameref{assmp:remainder} it follows that
  \begin{equation}
    \label{eqn:proof:borrow:second}
    \begin{split}
      \VV{u}^{T} \MM{A}_{rr}^{(\psi_{\min})} \VV{u} &\ge
      \VV{u}^{T}\left(\mm{I}\otimes\mm{D}^{T}\right)
      \left(\mm{H} \otimes \mm{H}\right)\MM{\Psi}_{\min}
      \left(\mm{I}\otimes\mm{D}\right) \VV{u}\\
      &=
      \sum_{j=0}^{N}
      {\left\{\mm{H}\right\}}_{jj}
      \VV{u}^{T}\left(\mm{e}_{j}\otimes\mm{D}^{T}\right)
      \mm{H}\MM{\Psi}_{\min}^{:j}
      \left(\mm{e}_{j}^{T}\otimes\mm{D}\right) \VV{u}\\
      &\ge
      \alpha h
      {\left(\vv{w}_{r}^{:0}\right)}^{T}
      \mm{H}\MM{\Psi}_{\min}^{:0}
      {\left(\vv{w}_{r}^{:0}\right)}^{T}
      +
      \alpha h
      {\left(\vv{w}_{r}^{:N}\right)}^{T}
      \mm{H}\MM{\Psi}_{\min}^{:N}
      {\left(\vv{w}_{r}^{:N}\right)}^{T};
    \end{split}
  \end{equation}
  since each term of the summation is positive, the last inequality follows by
  dropping all but the $j=0$ and $j=N$ terms of the summation. The result
  follows immediately by averaging~\eref{eqn:proof:borrow:first}
  and~\eref{eqn:proof:borrow:second}.
\end{proof}

We can now prove Theorem~\ref{thm:loc:PD} on the symmetric positive definiteness
of $\MM{M}$ as defined by~\eref{eqn:disc:system}.
\begin{proof}
  The structure of \eref{eqn:disc:system} directly implies that $\MM{M}$ is
  symmetric, in the remainder of the proof it is shown that $\MM{M}$ is also
  positive definite.

  We begin by recalling the definitions of $\MM{C}_{k}$ and $\mm{F}_{k}$
  in~\eref{eqn:full:disc} which allows us to write
  \begin{align}
    \label{eqn:Ck:Fk}
    \MM{C}_{k} = \mm{F}_{k} \mm{H}^{-1}\mm{\tau}_{k}^{-1}\mm{F}_{k}^{T} -
    \mm{G}_{k}^{T}\mm{H}^{-1}\mm{\tau}_{k}^{-1}\mm{G}_{k}.
  \end{align}
  Now considering the $\MM{M}$ weighted inner product we have that
  \begin{equation}\label{eqn:uMu}
  \begin{split}
    \VV{u}^{T}\MM{M}\VV{u}& =
    \VV{u}^{T}
    \left(\MM{A} + \sum_{k=1}^{4}\MM{C}_{k}\right)
    \VV{u}\\
    & =
    \VV{u}^{T}
    \left(
    \MM{\mathcal{A}} +
    \sum_{k=1}^{4}
    \mm{F}_{k}
    \mm{H}^{-1}\mm{\tau}_{k}^{-1}
    \mm{F}_{k}^{T}
    \right)
    \VV{u}
    \\
    & \quad +
    \VV{u}^{T}
    \left(
    \MM{A}_{rr}^{(\psi_{\min})}
    +
    \MM{A}_{ss}^{(\psi_{\min})}
    - \sum_{k=1}^{4}
    \mm{G}_{k}^{T}\mm{H}^{-1}\mm{\tau}_{k}^{-1}\mm{G}_{k}
    \right)
    \VV{u}.
  \end{split}
  \end{equation}
  Here we have used Lemma~\ref{app:lemma:split} to split $\MM{A}$.

  If $\mm{\tau}_{k} > 0$ then it follows for all $\VV{u}$ that
  \begin{align*}
    \VV{u}^{T} \mm{F}_{k} \mm{H}^{-1}\mm{\tau}_{k}^{-1} \mm{F}_{k}^{T}\VV{u}
    \ge 0.
  \end{align*}
  Additionally, if $\VV{u} = c \VV{1}$ for some constant $c \ne 0$ then it is a
  strict inequality since
  \begin{align*}
    \mm{F}_{k}^{T} \VV{1} = -\mm{H} \mm{\tau}_{k}\vv{1} \ne \vv{0}.
  \end{align*}
  Since by Lemma~\ref{app:lemma:split} the matrix $\MM{A}$ is symmetric positive
  semidefinite with $\nullspace(\MM{\mathcal{A}}) = \spanspace(\VV{1})$, this
  implies that the matrix
  \begin{align*}
    \MM{\mathcal{A}} +
    \sum_{k=1}^{4}
    \mm{F}_{k}
    \mm{H}^{-1}\mm{\tau}_{k}^{-1}
    \mm{F}_{k}^{T} \succ 0,
  \end{align*}
  that is the matrix is positive definite.
  To complete the proof all that remains is to show the remaining matrix
  in~\eref{eqn:uMu} is positive semidefinite, namely
  \begin{align*}
    \MM{A}_{rr}^{(\psi_{\min})}
    +
    \MM{A}_{ss}^{(\psi_{\min})}
    - \sum_{k=1}^{4}
    \mm{G}_{k}^{T}\mm{H}^{-1}\mm{\tau}_{k}^{-1}\mm{G}_{k} \succeq 0.
  \end{align*}

  Considering the quantity $\VV{u}^{T} \left(
    \MM{A}_{rr}^{(\psi_{\min})}
    +
    \MM{A}_{ss}^{(\psi_{\min})}\right)
  \VV{u}$ and using Lemma~\ref{app:lemma:bound:deriv} we can write:
  \begin{equation}
    \label{eqn:uAu}
    \begin{split}
    \VV{u}^{T}
    &
    \left(
    \MM{A}_{rr}^{(\psi_{\min})}
    +
    \MM{A}_{ss}^{(\psi_{\min})}
    \right)
    \VV{u}\\
    &\ge
    \frac{1}{2}
    \left(
      h\beta
      {\left(\vv{v}_{r}^{0:}\right)}^{T}
      \mm{H}\mm{\Psi}_{\min}^{0:}
      \vv{v}_{r}^{0:}
      +
      h \alpha
      {\left(\vv{w}_{s}^{0:}\right)}^{T}
      \mm{H}\mm{\Psi}_{\min}^{0:}
      \vv{w}_{s}^{0:}
    \right)
    \\ &
    \phantom{\ge}
    +
    \frac{1}{2}
    \left(
      h\beta
      {\left(\vv{v}_{r}^{N:}\right)}^{T}
      \mm{H}\mm{\Psi}_{\min}^{N:}
      \vv{v}_{r}^{N:}
      +
      h \alpha
      {\left(\vv{w}_{s}^{N:}\right)}^{T}
      \mm{H}\mm{\Psi}_{\min}^{N:}
      \vv{w}_{s}^{N:}
    \right)
    \\ &
    \phantom{\ge}
    +
    \frac{1}{2}
    \left(
      h\alpha
      {\left(\vv{w}_{r}^{:0}\right)}^{T}
      \mm{H}\mm{\Psi}_{\min}^{:0}
      \vv{w}_{r}^{:0}
      +
      h \beta
      {\left(\vv{v}_{s}^{:0}\right)}^{T}
      \mm{H}\mm{\Psi}_{\min}^{:0}
      \vv{v}_{s}^{:0}
    \right)
    \\ &
    \phantom{\ge}
    +
    \frac{1}{2}
    \left(
      h\alpha
      {\left(\vv{w}_{r}^{:N}\right)}^{T}
      \mm{H}\mm{\Psi}_{\min}^{:N}
      \vv{w}_{r}^{:N}
      +
      h \beta
      {\left(\vv{v}_{s}^{:N}\right)}^{T}
      \mm{H}\mm{\Psi}_{\min}^{:N}
      \vv{v}_{s}^{:N}
    \right).
  \end{split}
  \end{equation}
  Now considering the $k=1$ term of the last summation in~\eref{eqn:uMu} we have
  \begin{align}
    \label{eqn:uGHGu}
    \VV{u}^{T}
    \mm{G}_{1}^{T}\mm{H}^{-1}\mm{\tau}_{1}^{-1}\mm{G}_{1}
    \VV{u}
    =
    {\left(\mm{C}_{rr}^{0:}\vv{v}_{r}^{0:} + \mm{C}_{rs}^{0:}
    \vv{w}_{s}^{0:}\right)}^{T}
    \mm{H}\mm{\tau}_{1}^{-1}
    \left(\mm{C}_{rr}^{0:}\vv{v}_{r}^{0:} + \mm{C}_{rs}^{0:}
    \vv{w}_{s}^{0:}\right).
  \end{align}
  We now need to use the positive term related to face $1$ of~\eref{eqn:uAu}
  to bound the negative contribution from~\eref{eqn:uGHGu}. Doing this
  subtraction for face $1$ then gives:
  \begin{equation}\label{eqn:small:eig:mat}
    \begin{split}
      &\frac{1}{2}
      \left(
      h\beta
      {\left(\vv{v}_{r}^{0:}\right)}^{T}
      \mm{H}\mm{\Psi}_{\min}^{0:}
      \vv{v}_{r}^{0:}
      +
      h \alpha
      {\left(\vv{w}_{s}^{0:}\right)}^{T}
      \mm{H}\mm{\Psi}_{\min}^{0:}
      \vv{w}_{s}^{0:}
      \right)\\
      &\quad-
      {\left(\mm{C}_{rr}^{0:}\vv{v}_{r}^{0:} + \mm{C}_{rs}^{0:}
      \vv{w}_{s}^{0:}\right)}^{T}
      \mm{H}\mm{\tau}_{1}^{-1}
      \left(\mm{C}_{rr}^{0:}\vv{v}_{r}^{0:} + \mm{C}_{rs}^{0:}
      \vv{w}_{s}^{0:}\right)\\
      &=
      \begin{bmatrix}
        \vv{\hat{v}}_{r}^{0:}\\
        \vv{\hat{w}}_{s}^{0:}
      \end{bmatrix}^{T}
      \left(\mm{I}_{2\times2} \otimes \mm{H}\right)
      \begin{bmatrix}
        \mm{I} -
        {\left(\mm{\hat{C}}_{rr}^{0:}\right)}^{2} \mm{\tau}_{1}^{-1}
        &
        -\mm{\hat{C}}_{rr}^{0:}\mm{\tau}_{1}^{-1}\mm{\hat{C}}_{rs}^{0:}\\
        -\mm{\hat{C}}_{rs}^{0:}\mm{\tau}_{1}^{-1}\mm{\hat{C}}_{rr}^{0:}&
        \mm{I} -
        {\left(\mm{\hat{C}}_{rs}^{0:}\right)}^{2} \mm{\tau}_{1}^{-1}
      \end{bmatrix}
      \begin{bmatrix}
        \vv{\hat{v}}_{r}^{0:}\\
        \vv{\hat{w}}_{s}^{0:}
      \end{bmatrix}
      \\&
      =
      \sum_{j=0}^{N}
      H_{s}^{j}
      \begin{bmatrix}
        \hat{v}_{r}^{0j}\\
        \hat{w}_{s}^{0j}
      \end{bmatrix}^{T}
      \begin{bmatrix}
        1 -
        \frac{{\left(\hat{C}_{rr}^{0j}\right)}^{2}}{\tau^{j}_{1}}
        &
        -\frac{\hat{C}_{rr}^{0j}\hat{C}_{rs}^{0j}}{\tau^{j}_{1}}\\
        -\frac{\hat{C}_{rs}^{0j}\hat{C}_{rr}^{0j}}{\tau^{j}_{1}}&
        1 -
        \frac{{\left(\hat{C}_{rs}^{0j}\right)}^{2}}{\tau^{j}_{1}}
      \end{bmatrix}
      \begin{bmatrix}
        \hat{v}_{r}^{0j}\\
        \hat{w}_{s}^{0j}
      \end{bmatrix}.
    \end{split}
  \end{equation}
  In the above calculation we have used the fact that $\mm{H}$,
  $\mm{\tau}_{1}$, $\mm{C}_{rr}^{0:}$, and $\mm{C}_{rs}^{0:}$ are diagonal as
  well as made the following definitions:
  \begin{equation*}
    \begin{alignedat}{2}
      \hat{v}_{r}^{0j} &=
      v_{r}^{0j} \sqrt{\frac{1}{2}h\beta \Psi_{\min}^{0j}},& \quad
      \hat{C}_{rr}^{0j} &=
      C_{rr}^{0j}\sqrt{\frac{2}{h\beta\Psi_{\min}^{0j}} },\\
      \hat{w}_{s}^{0j} &=
      w_{s}^{0j} \sqrt{\frac{1}{2}h\alpha \Psi_{\min}^{0j}},& \quad
      \hat{C}_{rs}^{0j} &=
      C_{rs}^{0j}\sqrt{\frac{2}{h\alpha\Psi_{\min}^{0j}} }.
    \end{alignedat}
  \end{equation*}
  The eigenvalues of the matrix in~\eref{eqn:small:eig:mat} are:
  \begin{equation*}
    \mu_{1} = 1,\quad
    \mu_{2} = 1 - \frac{{\left(\hat{C}_{rr}^{0j}\right)}^{2} +
    {\left(\hat{C}_{rs}^{0j}\right)}^{2}}{\tau^{j}_{1}}.
  \end{equation*}
  The first eigenvalue $\mu_{1}$ is clearly positive and $\mu_{2}$ will be
  positive if:
  \begin{subequations}\label{eqn:tau}
  \begin{align}
    \tau_{1}^{j} >
    {\left(\hat{C}_{rr}^{0j}\right)}^2
    +
    {\left(\hat{C}_{rs}^{0j}\right)}^2
    =
    \frac{2{\left(C_{rr}^{0j}\right)}^2}{h\beta\Psi_{\min}^{0j}}
    +
    \frac{2{\left(C_{rs}^{0j}\right)}^2}{h\alpha\Psi_{\min}^{0j}}.
  \end{align}
  With such a definition of $\vv{\tau}_{1}$ all the terms
  in~\eref{eqn:small:eig:mat} are positive and thus for face $1$ the terms
  in~\eref{eqn:uMu} are positive. An identical argument holds for the other
  faces if:
  \begin{align}
    \tau_{2}^{j} &>
    \frac{2{\left(C_{rr}^{Nj}\right)}^2}{h\beta\Psi_{\min}^{Nj}}
    +
    \frac{2{\left(C_{rs}^{Nj}\right)}^2}{h\alpha\Psi_{\min}^{Nj}},\\
    \tau_{3}^{i} &>
    \frac{2{\left(C_{rs}^{i0}\right)}^2}{h\alpha\Psi_{\min}^{i0}}
    +
    \frac{2{\left(C_{ss}^{i0}\right)}^2}{h\beta\Psi_{\min}^{i0}},\\
    \tau_{4}^{i} &>
    \frac{2{\left(C_{rs}^{iN}\right)}^2}{h\alpha\Psi_{\min}^{iN}}
    +
    \frac{2{\left(C_{ss}^{iN}\right)}^2}{h\beta\Psi_{\min}^{iN}},
  \end{align}
  \end{subequations}
  and thus $\MM{M}$ is positive definite since $\VV{u}^{T}\MM{M}\VV{u} > 0$ for
  all $\VV{u} \ne \VV{0}$.
\end{proof}
\subsection{Proof of Theorem~\ref{thm:Neumann:loc:PD} (Positive Definiteness of
the Local Problem with Neumann Boundary
Conditions)}\label{sec:app:Neumann:loc:PD} Here we prove
Theorem~\ref{thm:Neumann:loc:PD} on the symmetric positive definiteness of
$\MM{M}$ with Neumann boundary conditions.
\begin{proof}
  We begin by considering
  \begin{align*}
    \VV{u}^{T}\MM{M}\VV{u}&
    =
    \VV{u}^{T}
    \left(\MM{A} + \MM{\mathcal{C}}_{1} + \MM{\mathcal{C}}_{2}
    + \MM{\mathcal{C}}_{3} + \MM{\mathcal{C}}_{4}\right) \VV{u},
  \end{align*}
  where we define the modified surface matrices $\MM{\mathcal{C}}_{k}$ to
  be
  \begin{align}
    \label{eqn:Ck:Neumann}
    \MM{\mathcal{C}}_{k}
    &=
    \MM{C}_{k} - \mm{F}_{k}
    \mm{H}^{-1}\mm{\tau}_{k}^{-1} \mm{F}_{k}^{T}
    =
    - \mm{G}_{k}^{T} \mm{H}^{-1}\mm{\tau}_{k}^{-1} \mm{G}_{k},
  \end{align}
  if face $k$ is a Neumann boundary and $\MM{\mathcal{C}}_{k} = \MM{C}_{k}$
  otherwise; see the definition of the modified
  $\MM{M}$ with Neumann boundary conditions \eref{eqn:M:Neumann} and
  \eref{eqn:Ck:Fk}. In the proof of Theorem~\ref{thm:loc:PD} it was shown that
  terms of the form of \eref{eqn:Ck:Neumann} combine with $\MM{A}$ is a way that
  is non-negative if $\mm{\tau}_{k}$ satisfy \eref{eqn:tau};
  see \eref{eqn:uAu} and following. Thus $\VV{u}^{T}\MM{M}\VV{u} \ge 0$ for all
  $\VV{u}$. The inequality will be strict for $\VV{u} \ne \VV{0}$ as long as one
  face is Dirichlet; the argument is that same as that made in the proof of
  Theorem~\ref{thm:loc:PD}.
\end{proof}
\subsection{Proof of Theorem~\ref{thm:coupled:PD} and
Corollary~\ref{cor:coupled:PD} (Positive Definiteness of the Global
Problem)}\label{sec:app:coupled:PD}

Proof of Theorem~\ref{thm:coupled:PD}
\begin{proof}
  Without loss of generality, we consider a two block mesh with Dirichlet
  boundary conditions with a single face $f \in \FF_{I}$ and
  assume that it is connected to face $k^{+}$ of block $B^{+}$ and face $k^{-}$
  of block $B^{-}$.  Solving for $\lambda_{f}$ in the global coupling equation
  \eref{eqn:global:face:couple} in terms of $\VV{u}_{B^{+}}$ and
  $\VV{u}_{B^{-}}$ gives
  \begin{equation*}
  \mm{\lambda}_{f}
  =
    \mm{D}_{f}^{-1} \left(
  \frac{1}{2}\mm{H}\left(
    \mm{\tau}_{f, B^{+}}
  -
    \mm{\tau}_{f, B^{-}}
  \right)
  \vv{\delta}_{f}
    -\mm{F}_{f,B^{+}}^{T} \VV{u}_{B^{+}}
    -\mm{F}_{f,B^{-}}^{T} \VV{u}_{B^{-}}
  \right).
  \end{equation*}
  Plugging this expression into the local problem
  \eref{eqn:full:disc}, gives
  \begin{equation}\label{eqn:glo}
    \begin{split}
    \left(\MM{A}_{B^{+}} - \mm{F}_{f,B^{+}} \mm{D}_{f}^{-1}
    \mm{F}_{f,B^{+}}^{T} + \sum_{k=1}^{4} \MM{C}_{k,B^{+}}
    \right)\VV{u}_{B^{+}}&\\
    - \mm{F}_{f,B^{+}} \mm{D}_{f}^{-1} \mm{F}_{f,B^{-}}^{T}\VV{u}_{B^{-}}
    &= \VV{q}_{B^{+} \setminus f},\\
    \left(\MM{A}_{B^{-}} - \mm{F}_{f,B^{-}} \mm{D}_{f}^{-1}
    \mm{F}_{f,B^{-}}^{T} + \sum_{k=1}^{4} \MM{C}_{k,B^{-}}
    \right)\VV{u}_{B^{-}}&\\
    - \mm{F}_{f,B^{-}} \mm{D}_{f}^{-1} \mm{F}_{f,B^{+}}^{T}\VV{u}_{B^{+}}
    &= \VV{q}_{B^{-} \setminus f}.
    \end{split}
  \end{equation}
  Here $\VV{q}_{B^{\pm} \setminus f}$ denotes $\VV{q}_{B^{\pm}}$
  (see~\eref{eqn:qtilde}) with the term dependent on $\VV{u}$ associated with
  face $f$ removed.  Using~\eref{eqn:Ck:Fk} which relates $\MM{C}_{f,B^{\pm}}$
  to $\mm{F}_{f,B^{\pm}}$ we have that
  \begin{equation*}
    \begin{split}
    \MM{C}_{f,B^{\pm}} - \mm{F}_{f,B^{\pm}} \mm{D}_{f}^{-1}
    \mm{F}_{f,B^{\pm}}^{T}& =
    \mm{F}_{f,B^{\pm}} \mm{H}^{-1} \left(
    \mm{\tau}_{f,B^{\pm}}^{-1}
    -
    {\left(
    \mm{\tau}_{f,B^{\pm}}
    +
    \mm{\tau}_{f,B^{-}}
    \right)}^{-1}
    \right)
    \mm{F}_{f,B^{\pm}}^{T}\\
    & \quad -\mm{G}_{k,B^{\pm}}^{T} \mm{H}^{-1}\mm{\tau}_{f,B^{\pm}}^{-1}\mm{G}_{k,B^{\pm}}.
    \end{split}
  \end{equation*}
  Plugging this into \eref{eqn:glo}, and rewriting the two
  equations as single system gives:
  \begin{subequations}
  \begin{align*}
    \left(
    \mm{\mathbb{A}} + \mm{\mathbb{F}}\; \mm{\mathbb{T}}\; \mm{\mathbb{F}}^{T}
    \right)
    \begin{bmatrix}
      \VV{u}_{B^{+}}\\
      \VV{u}_{B^{-}}
    \end{bmatrix}
    =
    \begin{bmatrix}
      \VV{q}_{B^{+}\setminus f}\\
      \VV{q}_{B^{-}\setminus f}
    \end{bmatrix},
  \end{align*}
  where we have defined the following matrices:
  \begin{align*}
    \mm{\mathbb{F}} &=
    \begin{bmatrix}
      \mm{H}^{1/2} \mm{F}_{f,B^{+}}
      &
      \mm{0}\\
      \mm{0}
      &
      \mm{H}^{1/2} \mm{F}_{f,B^{-}}
    \end{bmatrix},\\
    \mm{\mathbb{T}} &=
    \begin{bmatrix}
      \mm{\tau}_{f,B^{+}}^{-1}
      -
      {\left(
      \mm{\tau}_{f,B^{+}}
      +
      \mm{\tau}_{f,B^{-}}
      \right)}^{-1}&
      -{\left(
      \mm{\tau}_{f,B^{+}}
      +
      \mm{\tau}_{f,B^{-}}
      \right)}^{-1}
      \\
      -{\left(
      \mm{\tau}_{f,B^{+}}
      +
      \mm{\tau}_{f,B^{-}}
      \right)}^{-1}&
      \mm{\tau}_{f,B^{-}}^{-1}
      -
      {\left(
      \mm{\tau}_{f,B^{-}}
      +
      \mm{\tau}_{f,B^{-}}
      \right)}^{-1}&
    \end{bmatrix},
    \\
    \mm{\mathbb{A}} &=
    \begin{bmatrix}
      \mathbb{A}^{+}
      &
      \mm{0}\\
      \mm{0}
      &
      \mathbb{A}^{-}
    \end{bmatrix},\\
    \mm{\mathbb{A}^{\pm}} &=
       \MM{A}_{B^{\pm}}
       -\mm{G}_{k,B^{\pm}}^{T}
      \mm{H}^{-1}\mm{\tau}_{f,B^{\pm}}^{-1}\mm{G}_{k,B^{\pm}}
      + \sum_{\substack{k=1\\k\ne k^{\pm}}}^{4} \MM{C}_{k,B^{\pm}}.
  \end{align*}
  \end{subequations}
  The matrix $\mm{\mathbb{A}}$ is block diagonal, and each of the blocks was
  shown in the proof of Theorem~\ref{thm:loc:PD} to be symmetric positive
  semidefinite.  Thus, if $\mm{\mathbb{T}}$ is symmetric positive semidefinite,
  then the whole system is symmetric positive semidefinite.  Since
  $\mm{\tau}_{f,B^{\pm}}$ are diagonal, the eigenvalues $\mm{\mathbb{T}}$ are
  the same as the eigenvalues of the
  $2\times2$ systems
  \begin{align*}
    \mm{\mathbb{T}}^{j}
    &=
    \begin{bmatrix}
      \frac{1}{\tau^{j}_{f,B^{+} }}
      -
      \frac{1}{ \tau^{j}_{f,B^{+}} + \tau^{j}_{f,B^{-}} }&
      -\frac{1}{ \tau^{j}_{f,B^{+}} + \tau^{j}_{f,B^{-}} }
      \\
      -\frac{1}{ \tau^{j}_{f,B^{+}} + \tau^{j}_{f,B^{-}} }&
      \frac{1}{\tau^{j}_{f,B^{-} }}
      -
      \frac{1}{ \tau^{j}_{f,B^{-}} + \tau^{j}_{f,B^{-} }}
    \end{bmatrix}\notag\\
    &=
    \frac{1}{ \tau^{j}_{f,B^{+}} + \tau^{j}_{f,B^{-}} }
    \begin{bmatrix}
      \frac{\tau^{j}_{f,B^{-} }}{\tau^{j}_{f,B^{+} }} & -1 \\
      -1 & \frac{\tau^{j}_{f,B^{+} }}{\tau^{j}_{f,B^{-} }}
    \end{bmatrix},
  \end{align*}
  for each $j = 0$ to $N_{f}$ (number of points on the face). The eigenvalues of
  $\mm{\mathbb{T}}^{j}$ are
  \begin{align*}
    \mu_{1} &= 0,&
    \mu_{2} &= \frac{\tau_{f,B^{+}}^2 +
    \tau_{f,B^{-}}^2}{\tau_{f,B^{+}}\tau_{f,B^{-} }},
  \end{align*}
  which shows that $\mm{\mathbb{T}}^{j}$ and that $\mm{\mathbb{T}}$ are positive
  semidefinite as long as $\tau_{f,B^{\pm}}^{j} > 0$.

  An identical argument holds for each interface $f \in \FF$, thus the
  interface treatment guarantees the global system of equations is symmetric
  positive semidefinite. Positive definiteness results as long as one of the
  faces of the mesh is a Dirichlet boundary since only the constant state over
  the entire domain is in the $\nullspace(\MM{A}_{B})$ for all $B\in\BB$ and
  this is removed as long as some face of the mesh has a Dirichlet
  boundary condition; see proof of Theorem~\ref{thm:loc:PD}.
\end{proof}

Proof of Corollary~\ref{cor:coupled:PD}
\begin{proof}
  Begin by noting that
  \begin{align*}
    \begin{bmatrix}
      \mm{\bar{M}} & \mm{\bar{F}}\\
      \mm{\bar{F}}^{T} & \mm{\bar{D}}
    \end{bmatrix}
    =
    \begin{bmatrix}
      \mm{\bar{I}} & \mm{\bar{F}}\mm{\bar{D}}^{-1}\\
      \mm{\bar{0}} & \mm{\bar{I}}
    \end{bmatrix}
    \begin{bmatrix}
      \mm{\bar{M}} - \mm{\bar{F}} \mm{\bar{D}}^{-1}\mm{\bar{F}}^{T} & \mm{\bar{0}}\\
      \mm{\bar{0}} & \mm{\bar{D}}
    \end{bmatrix}
    \begin{bmatrix}
      \mm{\bar{I}} & \mm{\bar{0}}\\
      \mm{\bar{D}}^{-1}\mm{\bar{F}}^{T} & \mm{\bar{I}}
    \end{bmatrix}.
  \end{align*}
  By Theorem~\ref{thm:coupled:PD} and structure of $\mm{\bar{D}}$ the block
  diagonal center matrix is symmetric positive definite. Since the outer two
  matrices are the transposes of one another, it immediately follows that the
  global system matrix is symmetric positive definite.

  Since the global system matrix and $\mm{\bar{M}}$ are symmetric positive
  definite, symmetric positive definiteness of the Schur complement of the
  $\mm{\bar{M}}$ block follows directly from the decomposition
  \begin{align*}
    \begin{bmatrix}
      \mm{\bar{M}} & \mm{\bar{F}}\\
      \mm{\bar{F}}^{T} & \mm{\bar{D}}
    \end{bmatrix}
    =
    \begin{bmatrix}
      \mm{\bar{I}} & \mm{\bar{0}}\\
      \mm{\bar{F}}^{T}\mm{\bar{M}}^{-1} & \mm{\bar{I}}
    \end{bmatrix}
    \begin{bmatrix}
      \mm{\bar{M}} & \mm{\bar{0}}\\
      \mm{\bar{0}} & \mm{\bar{D}} - \mm{\bar{F}}^{T} \mm{\bar{M}}^{-1}\mm{\bar{F}}
    \end{bmatrix}
    \begin{bmatrix}
      \mm{\bar{I}} & \mm{\bar{M}}^{-1}\mm{\bar{F}}\\
      \mm{\bar{0}} & \mm{\bar{I}}
    \end{bmatrix}.
  \end{align*}
\end{proof}

\bibliographystyle{spmpscinat}
\bibliography{refs}

\end{document}